\documentclass[a4paper]{amsart}

\usepackage[utf8]{inputenc}
\usepackage{amsmath}
\usepackage{amssymb}
\usepackage{esint} 
\usepackage{hyperref}

\newcommand{\R}{\mathbb{R}}
\newcommand{\e}{\varepsilon}
\renewcommand{\H}{\mathcal{H}}
\newcommand{\E}{\mathcal{E}}
\renewcommand{\L}{\mathcal{L}}
\newcommand{\lag}{\mathcal{\ell}}

\newcommand{\tail}{\operatorname{Tail}}
\newcommand{\supp}{\operatorname{supp}}

\newtheorem{theorem}{Theorem}[section]
\newtheorem{proposition}{Proposition}[section]
\newtheorem{lemma}{Lemma}[section]
\newtheorem{corollary}{Corollary}[section]
\theoremstyle{definition}
\newtheorem{definition}{Definition}
\newtheorem{remark}{Remark}[section]

\title{Harnack's inequality for parabolic nonlocal equations}
\author{Martin Str\"omqvist}
\address{Martin Str{\"o}mqvist\\Department of Mathematics, Uppsala University\\
SE-751 06 Uppsala, Sweden}
\email{martin.stromqvist@math.uu.se}


\begin{document}

\begin{abstract}
The main result of this paper is a nonlocal version of Harnack's inequality for a class of parabolic nonlocal equations. We additionally establish a weak Harnack inequality as well as local boundedness of solutions. None of the results require the solution to be globally positive.  
\end{abstract}

\maketitle

\section{Introduction and main results}

The purpose of this paper is to establish a Harnack inequality for weak solutions to equations of the type 
\begin{equation}\label{maineq}
		\partial_t u(x,t) + \L u(x,t) = 0\quad\text{in }\R^n\times(0,T),
\end{equation}
where 
\[
\L u(x,t) = \text{P.V.}\int_{\R^n}(u(x,t)-u(y,t))K(x,y,t)dy. 
\]
We assume that $K$ is symmetric with respect to $x$ and $y$ and satisfies, for some $\Lambda\ge 1$ and $s\in(0,1)$, the ellipticity condition   
\begin{equation}\label{ellipticity}
\frac{\Lambda^{-1}}{|x-y|^{n+2s}}\le K(x,y,t) \le \frac{\Lambda}{|x-y|^{n+2s}},
\end{equation}
uniformly in $t\in (0,T)$. When 
\[
K(x,y,t) = \frac{C(n,s)}{|x-y|^{n+2s}}, 
\]
for appropriate choice of $C(n,s)$, $\L$ is the fractional Laplacian and \eqref{maineq} is called the fractional heat equation. Equations of the type \eqref{maineq} appear for instance in the study of Levy processes as well as in signal and image processing. 

\subsection{Notation}
Our estimates feature a nonlocal quantity defined below, called the parabolic tail. The time dependence in the parabolic tails is one of the main difficulties that arise in the parabolic setting compared to the elliptic.  
\begin{definition}\label{def_tail}
If $v$ is a measurable function on $\R^n\times(0,T)$, and $x_0\in \R^n$, $r>0$, $0<t_1<t_2<T$, the parabolic tail of $v$ with respect to $x_0,r,t_1,t_2$ is defined by 
\begin{equation}\label{theTail}
\tail(v;x_0,r,t_1,t_2) = \frac{r^{2s}}{t_2-t_1}\int_{t_1}^{t_2}\int_{\R^n\setminus B_r(x_0)}\frac{|v(x,t)|}{|x-x_0|^{n+2s}}dxdt.
\end{equation}
We also define the parabolic supremum tail of $v$ with respect to $x_0,r,t_1,t_2$ by 
\begin{equation}\label{thesupTail}
\tail_\infty(v;x_0,r,t_1,t_2) = r^{2s}\sup_{t_1<t<t_2}\int_{\R^n\setminus B_r(x_0)}\frac{|v(x,t)|}{|x-x_0|^{n+2s}}dx.
\end{equation}
\end{definition}

For $x_0\in \R^n$ and $r>0$, $B_r(x_0)$ denotes the ball in $\R^n$ of radius $r$ and center $x_0$. When the point $x_0$ is clear from the context we simply write $B_r$. For $t_0\in (r^{2s},T-r^{2s})$, we define the parabolic cylinders  
\begin{align*}
&U^-(r) = U^-(x_0,t_0,r) = B_r(x_0)\times(t_0-r^{2s},t_0),\\
&U^+(r) = U^+(x_0,t_0,r) = B_r(x_0)\times(t_0,t_0 + r^{2s}).
\end{align*}
We denote the positive and negative parts of a function $v(x,t)$ by 
\[
v_+(x,t) = \max\{v(x,t),0\},\quad v_-(x,t) = \max\{-v(x,t),0\}.
\]
The measure $K(x,y,t)dxdy$ occurs frequently in our proofs and, for the sake of brevity, we shall often use the notation 
\[
d\mu = d\mu(x,y,t) = K(x,y,t)dxdy. 
\]
Throughout the paper, $C$ will denote a generic positive constant depending only on $n,s,\Lambda$. 

\subsection{Main results and overview of related literature}
Theorems \ref{Harnack}-\ref{thm_localboundfinal} below are the main results of the paper. Note that the solution is not required to be nonnegative globally.  To the authors best knowledge, they are new even for the fractional heat equation. For operators of the type in \eqref{maineq}, that may depend on time and possess no regularity other than the ellipticity condition \eqref{ellipticity}, Theorem \ref{Harnack}, \ref{thm_localbound} and \ref{thm_localboundfinal} seem to be new even in the context of globally positive solutions. 
\begin{theorem}[Harnack inequality]\label{Harnack} 
Let $0<r<R/2$, let $t_0>r^{2s}$ and let 
\[t_1 = t_0+2r^{2s}-\alpha(r/2)^{2s},\quad\text{for some }\alpha\in (1,2^{2s}). \]
Suppose that $t_1<T$ and that $u$ is a solution to \eqref{maineq} such that  
\[
u\ge 0\text{ in } B_R(x_0)\times(t_0-r^{2s},t_1). 
\]
Then
\begin{align*}
\sup_{U^-(x_0,t_0,r/2)}u\le C\left(\inf_{U^-(x_0,t_1,r/2)}u + \left(\frac{r}{R}\right)^{2s}\tail(u_-;x_0,R,t_0-r^{2s},t_1)\right),  
\end{align*}
where $C$ depends on $n,s,\Lambda$ and $\alpha$.


\end{theorem}

\begin{theorem}[Weak Harnack inequality]\label{weakHarnack}
Suppose that $u$ is a supersolution to \eqref{maineq} such that 
\[
u\ge 0 \text{ in } B_R(x_0)\times(t_0-2r^{2s},t_0+2r^2), \quad r<R/2.
\]
Then 
\begin{align*}
&\fint_{B_r(x_0)\times(t_0-2r^{2s},t_0-r^{2s})}udxdt\le C\inf_{B_r(x_0)\times(t_0+r^{2s},t_0+2r^{2s})}u\\
&\qquad + C\left(\frac{r}{R}\right)^{2s}\tail_\infty(u_-;x_0,R,t_0-2r^{2s},t_0+2r^{2s}).\notag
\end{align*}
\end{theorem}

The next two theorems concern local boundedness of subsolutions. 
\begin{theorem}\label{thm_localbound}
Suppose that $u$ is a subsolution to \eqref{maineq}.  
Then for any $x_0\in\R^n$, $r>0$, $t_0\in (r^{2s},T)$, $\theta\in(0,1)$ and any $\delta\in(0,1)$, there exist positive constants $C(\delta) = C(\delta,n,\Lambda,s)$ and $m = m(n,s)$, such that 
\begin{align*}
\sup_{U^-(x_0,t_0,\theta r)}u &\le  \frac{C(\delta)}{(1-\theta)^{m}}\fint_{U^-(x_0,t_0,r)}u_+dxdt\\
&+ \delta \tail(u_+;x_0,r,t_0-r^{2s},t_0).\notag
\end{align*}
\end{theorem}

\begin{theorem}\label{thm_localboundfinal}
Suppose that $u$ is a subsolution to \eqref{maineq} such that 
\[
u\ge 0\text{ in }B_{R}(x_0)\times(t_0-r^{2s},t_0),\quad r<R/2,
\]
where $t_0\in(r^{2s},T)$.  
Then for any $\theta\in(0,1)$ and any $\delta\in(0,1)$, there exist positive constants $C(\delta) = C(\delta,n,\Lambda,s)$ and $m= m(n,s)$ such that 
\begin{align*}
\sup_{U^-(x_0,t_0,\theta r)}u& \le \frac{C(\delta)}{(1-\theta)^{m}}\fint_{U^-(x_0,t_0,r)}u_+dxdt\\
& + \delta\left(\frac{r}{R}\right)^{2s}\tail(u_-;x_0,R,t_0- r^{2s},t_0). \notag
\end{align*}
\end{theorem}

The tail of the negative part of the solution enters in a crucial way. If $u$ is assumed to be nonnegative throughout $\R^n$ for all relevant times, the results are analogous to the corresponding theorems for local equations. For instance, Theorem \ref{thm_localboundfinal} asserts in this situation that the solution is locally bounded in terms of its local $L^1$-norm only. In Theorem \ref{weakHarnack} the supremum version of the tail, $\tail_\infty$, is used rather than $\tail$. We will see later in Lemma \ref{tail_inf_le_tail} and Corollary \ref{tail_inf_le_tail_minus} that $\tail_\infty(v;x_0,t_0,t_2)$ can be estimated in terms of $\tail(v;x_0,t_1,t_2)$ if $t_1<t_0$ and $v$ is either the positive part of a subsolution or the negative part of a supersolution. The technique that we use for this estimate requires us to work with global solutions. In fact, this is the only reason for us to consider global solutions. Under the hypothesis that Lemma \ref{tail_inf_le_tail} and Corollary \ref{tail_inf_le_tail_minus} hold, Theorems \ref{Harnack}-\ref{thm_localboundfinal} hold for functions that are solutions only locally. 

For solutions to elliptic equations \ $\L u=0$ in $B_r$, that are nonnegative in $B_R\supset B_r$, the following Harnack inequality holds:  
\begin{equation}\label{ellipticHarnack}
\sup_{B_{r/2}}u\le C\left(\inf_{B_{r/2}}u + \left(\frac{r}{R}\right)^{2s}\tail(u_-;x_0,R)\right), 
\end{equation}
where 
\[
\tail(u_-;x_0,R) = R^{2s}\int_{\R^n\setminus B_R(x_0)}\frac{u_-dx}{|x-x_0|^{n+2s}}. 
\]
The Harnack inequality \eqref{ellipticHarnack} is due to Kassmann, who proved it for the fractional laplacian, see \cite{kassmann}, \cite{kassmann2}. In \cite{kassmann2} a counterexample is provided that shows that the tail-contribution in \eqref{ellipticHarnack} is actually necessary. The Harnack inequality \eqref{ellipticHarnack} was later proven to hold for solutions to more general fractional operators of $p$-Laplace type, with a suitably adjusted tail-term if $p\neq 2$, see \cite{DKPpmin} and \cite{DKPharnack} by Di Castro, Kuusi and Palatucci. In the papers \cite{DKPpmin}, \cite{DKPharnack}, which have to be considered the state of the art of the elliptic theory, the authors additionally prove local boundedness and H\"older continuity of solutions.     

In the parabolic context, Harnack's inequality has, to the author's best knowledge only been proved for solutions that are globally positive, using representation formulas in terms of the heat kernel. 
In the probabilistic setting, Harnack inequalitys have been established using the connection between stochastic processes and equations similar to \eqref{maineq}. See for example \cite{BBK} and the references therein. In \cite{BSV}, Bonforte, Sire and Vazquez develop an optimal existence and uniqueness theory for the Cauchy problem for the fractional heat equation  posed in $\R^n$. For globally positive solutions to the fractional heat equation, they prove a Harnack inequality in which the usual timelag present in parabolic Harnack inequalities does not occur. This is due to the fact that the fractional heat kernel is not of Gaussian form. Thus the time lag present in Theorem \ref{Harnack} and \ref{weakHarnack} does not seem to be necessary.  
  
Felsinger and Kassmann \cite{FelKass} prove a weak Harnack inequality and H\"older continuity for weak solutions to \eqref{maineq} that are globally positive. They work with a class of kernels satisfying slightly weaker growth conditions than \eqref{ellipticity}. Due to the assumption of global positivity, the nonlocal term involving the negative part of the solution (the tail term), that normally occur in such estimates, is not present. In \cite{SchwKass}, Schwab and Kassmann prove results similar to those in \cite{FelKass}, but with $a(t,x,y)d\mu(x,y)$ in place of $K(t,x,y)dxdy$, merely assuming that $\mu$ is a measure, not necessarily absolutely continuous w.r.t.\ Lebesgue measure, that satisfies certain growth conditions. It should also be mentioned that the conditions imposed on the kernels/measures in \cite{FelKass} and \cite{SchwKass} are in general not sufficient to prove a Harnack inequality. This is due to a result by Bogdan and Sztonyk \cite{BogSzton} that prove sharp conditions on the kernel for a Harnack inequaity to hold (in the elliptic setting). 

In \cite{MS} by the author, local boundedness of solutions to degenerate nonlocal parabolic equations of $p-$Laplace type is proved. The proof is valid for $p>2$ and not $p=2$ that is considered in this paper. The bounds established in \cite{MS} depend on the supremum-version of the tail \eqref{thesupTail}. In that sense they are weaker than those established in the present paper.  Otherwise there seem to exist no previous theory of local boundedness, i.e.\ results in the spirit of Theorem \ref{thm_localbound} - \ref{thm_localboundfinal}, for parabolic nonlocal equations. 

In \cite{CCV}, Caffarelli, Chan and Vasseur study parabolic nonlocal, nonlinear equations of quadratic growth in all space. They prove that solutions are bounded and H\"older continuous as soon as the initial data is in $L^2$. Their results apply to the situation of the present paper. Thus, if we specify initial data $u_0\in L^2(\R^n)$ at time $t=0$ for the equation \eqref{maineq}, its solution will be H\"older continuous. 

\subsection{Outline of the paper}

In section \ref{prel} we cast $\L$ as an operator in divergence form, and introduce weak sub- and supersolutions to equation \eqref{maineq}, as well as some of their properties. We also establish Caccippoli inequalities that are crucial for the proofs of Theorems \ref{Harnack} -\ref{thm_localboundfinal}. Finally we provide estimates for the parabolic tails introduced in Definition \ref{def_tail}. An indispensable tool here is the fact that the weight function appearing in the definition of the tails behaves almost like an eigenfunction for the operator $\L$. This result first appeared in \cite{BV} and was used in \cite{BSV}. 
Section \ref{sec_weakHarnack} is devoted to the proof of Theorem \ref{weakHarnack}, the weak Harnack inequality. The structure of the proof follows Mosers original ideas. Theorem \ref{weakHarnack} was proved under the additional hypothesis that $u\ge 0$ in $\R^n\times(t_0-r^{2s},t_0+r^{2s})$ in \cite{FelKass}. In section 3 we prove Theorem \ref{thm_localbound} and \ref{thm_localboundfinal}. The proof uses De Giorgi's approach together with the estimates for the estimates for parabolic tails proved in section 2. Finally, in section 4 we obtain Harnack's inequality in a standard way using the previous results.

\section{Preliminaries and tools}\label{prel}

For a domain $D\subseteq\R^n$, 
the Sobolev space $H^s(D)$ consists of all functions $f\in L^2(D)$ such that the semi-norm 
\[
[f]_{H^s(D)} = \left(\int_{D} \int_{D}\frac{|f(x)-f(y)|^2dxdy}{|x-y|^{n+2s}}\right)^{\frac{1}{2}}
\]
is finite. The norm of $f\in H^{s}(D)$ is given by 
\[
\|f\|_{H^{s}(D)} = [f]_{H^s(D)} + \|f\|_{L^2(D)}. 
\]
The dual space of $H^s(D)$ is denoted $H^{-s}(D)$. We write $\langle\cdot,\cdot\rangle$ for the duality pairing between $H^s(D)$ and $H^{-s}(D)$. 
The parabolic Sobolev space $L^2(0,T;H^{s}(D))$ is the set of measurable functions on $(0,T)\times D$ such that the norm 
\[
\|f\|_{L^2(0,T;H^{s}(D))} = \left(\int_0^T\|f(\cdot,t)\|_{H^s(D)}^2dt\right)^{\frac12},
\]
is finite. Its dual space, $L^2(0,T;H^{-s}(D))$, is defined analogously. 
\subsection{Weak solutions}
We treat $\L$ as an operator in divergence form. Let 
\[
\E(u,v,t) = \int_{\R^n}\int_{\R^n}(u(x,t)-u(y,t))(v(x,t)-v(y,t))K(x,y,t)dxdy.
\]
Then if $u$ and $v$ are sufficiently smooth, 
\[
\int_{\R^n}\L u(x,t)v(x,t)dx = 2\E(u,v,t). 
\]
Thus, in order for the definition of weak solution given below to be consistent with \eqref{maineq}, we need to use $\frac12K$ rather than $K$ in the definition of $\E$. 
\begin{definition}
We say that $u$ is a weak subsolution (supersolution) to \eqref{maineq} if 
\begin{align}
&\int_0^T\langle\partial_tu,\phi\rangle dt + \int_0^T\E(u,\phi,t)dt \le 0\;(\ge 0),
\end{align}
for all nonnegative $\phi\in \H = \{v\in L^2(0,T;H^{s}(\R^n)):\partial_tv\in L^2(0,T;H^{-s}(\R^n))\}$ such that $\phi(\cdot,0) = \phi(\cdot,T) = 0$. 
Such a function will be referred to as a test function. A solution to \eqref{maineq} is a function that is both a subsolution and a supersolution. 
\end{definition}
When $\phi$ has a time derivative in the classical sense, a weak subsolution (supersolution) to \eqref{maineq} satisfies 
\begin{align}\label{weak2}
&-\int_0^T\int_{\R^n} u\partial_t\phi dx dt + \int_0^T\E(u,\phi,t)dt \le 0\;(\ge 0).
\end{align}
If we additionally specify initial data $u(x,0) = u_0(x)\in L^2(\R^n)$, a unique weak solution can be constructed using Galerkin's method. We also refer to \cite{BSV} for a much more advanced theory of existence and uniqueness for the fractional heat equation.  

\begin{remark}\label{rmk1}
We here briefly explain how to regularize test functions in a way that enables us to work with solutions as though they were bounded and smooth in $t$. 
In order not to overburden our proofs, we will not do this explicitly later on but refer to this remark instead. 
If $f\in H^s(D)$, then $f_+(x) = \max\{f(x),0\}$ belongs to $H^s(D)$ and 
\[
[f_+]_{H^s(D)}\le [f]_{H^s(D)}.
\]
This is simply due to the fact that for any $a,b\in \R$, $|a_+-b_+|\le |a-b|$. Since $\min\{a,M\} = M-(M-a)_+$, we see that a truncation does not increase the semi-norm in $H^s(D)$: 
\begin{equation}\label{max}
[\min\{f,M\}]_{H^s(D)}\le [f]_{H^s(D)},\quad\text{for any }M\in \R.
\end{equation}
Similarly, we have 
\begin{equation}\label{min}
[\max\{f,M\}]_{H^s(D)}\le [f]_{H^s(D)}, \quad\text{for any }M\in \R.
\end{equation}
Let $\zeta\in C_c^\infty(-1/2,1/2)$ be a non negative function such that 
$\zeta(t) = \zeta(-t)$ and $$\int_{-1/2}^{1/2}\zeta dt = 1.$$ For $h>0$, set $\zeta_h(t) = \zeta(t/h)h^{-1}$. If, $a<b$, $f\in L^1(a,b)$, $(\alpha,\beta)\subset (a+h/2,b-h/2)$ and $t\in(\alpha,\beta)$, let 
\[
f_h(t) = \int_a^b f(s)\zeta_h(t-s)ds. 
\]
Then $f_h$ is smooth on $(\alpha,\beta)$ and $\lim_{h\to0}f_h(t) = f(t)$ for a.e.\ $t\in (a,b)$. If $g(t)\in L^1(a,b)$, it is not hard to check, using the symmetry of $\zeta$, that 
\begin{equation}\label{ipp}
\int_\alpha^\beta f(t)\partial_tg_h(t)dt = -\int_\alpha^\beta \partial_tf_h(t)g(t)dt. 
\end{equation}
When deriving estimates from \eqref{weak2}, it may be assumed that $u(x,t)$ is bounded and differentiable in $t$ thanks to \eqref{max}, \eqref{min} and \eqref{ipp}. We would typically like to use a test function of the form $\phi(x,t) = u^p\psi(x)\eta(t)$ in \eqref{weak2} which is not in general possible. However, $\phi = ((\min\{u,M\})_h^p\psi\eta)_h$ is a valid test function for $p\ge 1$. If $p<1$ we need to replace $\min$ by $\max$. If $\eta$ has compact support in $(0,T)$, then by \eqref{ipp}, 
\begin{equation}
-\int_0^T\int_{\R^n}u\partial_t\phi dxdt = \int_0^T\int_{\R^n}\partial_t(\min\{u,M\})_h(\min\{u,M\})_h^p\psi\eta dxdt. 
\end{equation}
Thus we may work qualitatively with solutions as though they were bounded (above or below) and smooth in $t$ (with parameters $M,h$) as long as our estimates do not depend upon $M$ or $h$ and send $h\to0$ and $M\to\infty$ in the end. 

\end{remark}

\begin{remark}\label{rmk2}
If $u$ is a weak subsolution (supersolution) to \eqref{maineq} and $[t_1,t_2]\subset (0,T)$, then 
\begin{align}\label{subintervalipp}
&\int_{\R^n} u(x,t_2)\phi(x,t_2) dx - \int_{\R^n} u(x,t_1)\phi(x,t_1) dx\\
&-\int_{t_1}^{t_2}\int_{\R^n}u\partial_t\phi dxdt+ \int_{t_1}^{t_2}\E(u,\phi,t)dt \le 0\;(\ge 0), \notag
\end{align}
for all non negative smooth test functions $\phi$. To see this, let $\eta_j$ be a sequence of smooth, non negative functions on $\R$, with compact support in $(0,T)$, such that $\lim_{j\to\infty}\eta_j(t)=\chi_{(t_1,t_2)}$ a.e.\ Testing with $\phi\eta_j$ then and integrating by parts gives 
\begin{align}\label{subinterval}
&\int_{0}^{T}\int_{\R^n}\phi\eta_j\partial_tudxdt + \int_{0}^{T}\int_{\R^n}\int_{\R^n}\E(u,\phi\eta_j,t)dt\le 0\;(\ge 0). 
\end{align}
We recall that it may be assumed that $\partial_tu$ exists by Remark \ref{rmk1}. By Lebesgue's dominated convergence theorem, taking $j\to\infty$ 
in \eqref{subinterval} results in 
\begin{align}\label{dtu}
&\int_{t_1}^{t_2}\int_{\R^n}\phi\partial_tudxdt + \int_{t_1}^{t_2}\int_{\R^n}\int_{\R^n}\E(u,\phi,t)dt \le 0\;(\ge 0). 
\end{align}
Then \eqref{subintervalipp} follows after integrating by parts. In \eqref{subintervalipp} and \eqref{dtu}, $\phi$ does not need to have compact support in $(t_1,t_2)$. 

\end{remark}

The next lemma is a standard fact for local equations, but we have found no proof in the literature for nonlocal equations. 

\begin{lemma}\label{lemma_subsol}
If $u$ is a weak subsolution to \eqref{maineq}, then $u_+$ is a weak subsolution to \eqref{maineq}.  
\end{lemma}

\begin{proof}

Let $z_j(\tau)$ be a smooth, convex approximation of $\tau_+$ such that \[z_j(\tau)=0\text{ if }\tau\le -1/j,\;  z_j(\tau),\;z_j'(\tau)>0\text{ if }\tau>-1/j\text{ and }|z_j'|\le C,\; |z_j''|\le C(j).
\] 
Let $\zeta_j(x,t) = z_j(u(x,t))$ and let $\zeta_j'(x,t) = z_j'(u(x,t))$. We also set 
\begin{equation}
u_{j,+}(x,t) = \max\{u(x,t),-1/j\} = 
\left\{\begin{array}{l}
u(x,t)\text{ if }\zeta_j'(x,t)>0,\\
-1/j\text{ if }\zeta_j'(x,t)=0.
\end{array}\right.
\end{equation}
Let $\phi$ be a nonnegative, bounded test function. By appealing to Remark \ref{rmk1}, it is easy to verify that $\phi\zeta'_j$ is an admissible test function. 
Using $\zeta_j'\phi$ as a test function in \eqref{weak2} we obtain
\begin{align*}
&\int_{0}^{T}\int_{\Omega}\partial_tu\zeta_j'\phi dxdt + \int_{0}^{T}\int_{\R^n}\int_{\R^n}(u(x,t)-u(y,t))(\phi\zeta_j'(x,t)-\phi\zeta_j'(y,t))d\mu dt\\
& = I_{1,j}+I_{2,j} \le 0. 
\end{align*}
We may write $I_{1,j}$ as 
\begin{equation}\label{timeuplus}
I_{1,j} = \int_{0}^{T}\int_{\Omega}\phi\partial_tz_j(u)dxdt\to I_{1} = \int_{0}^{T}\int_{\Omega}\phi\partial_t u_+dxdt, \quad\text{as }j\to\infty.
\end{equation}
We next estimate the integrand of $I_{2,j}$ under the assumption that $u(x,t)>u(y,t)$. 
If $\zeta_j'(x,t)=0$, then $\zeta_j'(y,t)=0$ since $\zeta'$ is monotone nondecreasing. Hence the integrand of $I_{2,j}$ vanishes for such $(x,y,t)$. 
If $\zeta_j'(y,t)>0$, then 
\begin{align*}
&(u(x,t)-u(y,t))(\zeta_j'(x,t)\phi(x,t)-\zeta_j'(y,t)\phi(y,t))\\
& =(u_{j,+}(x,t)-u_{j,+}(y,t))(\zeta_j'(x,t)\phi(x,t)-\zeta_j'(y,t)\phi(y,t))\\
&\ge(u_{j,+}(x,t)-u_{j,+}(y,t))\zeta_j'(x,t)(\phi(x,t)-\phi(y,t)).
\end{align*}
If $\zeta_j'(y,t)=0$ and $\zeta_j'(x,t)>0$, then 
\begin{align*}
& (u(x,t)-u(y,t))(\zeta_j'(x,t)\phi(x,t)-\zeta_j'(y,t)\phi(y,t))\\
& = (u(x,t)-u(y,t))\zeta_j'(x,t)\phi(x,t)\\
&\ge (u_{j,+}(x,t)-u_{j,+}(y,t))\zeta_j'(x,t)\phi(x,t)\\
&\ge (u_{j,+}(x,t)-u_{j,+}(y,t))\zeta_j'(x,t)(\phi(x,t)-\phi(y,t)). 
\end{align*}
We have thus shown that if $u(x,t)>u(y,t)$, 
\begin{align}
&(u(x,t)-u(y,t))(\zeta_j'(x,t)\phi(x,t)-\zeta_j'(y,t)\phi(y,t))\label{phi(x)}\\
&\ge (u_{j,+}(x,t)-u_{j,+}(y,t))\zeta_j'(x,t)(\phi(x,t)-\phi(y,t)).\notag
\end{align}
If $u(x,t)<u(y,t)$, we obtain the same estimate by interchanging the roles of $x$ and $y$.
By dominated convergence, we obtain from \eqref{phi(x)}
\begin{align*}
&\liminf_{j\to\infty}I_{2,j}\\
&\ge \int_{0}^{T}\int_{\R^n}\int_{\R^n}K(x,y,t)(u_{+}(x,t)-u_{+}(y,t))(\phi(x,t)-\phi(y,t))dxdydt\\
& = \int_{0}^{T}\E(u_+,\phi,t)dt.
\end{align*}
In combination with \eqref{timeuplus}, this gives 
\[
\int_{0}^{0}\int_{\Omega}v\partial_t u_+dxdt + \int_0^T\E(u_+,\phi,t)dt\le 0,
\]
for all bounded, nonnegative test functions $\phi$, and by a standard approximation argument, all nonnegative test functions $\phi$.  

\end{proof}

\subsection{Sobolev inequalities}

For the basic properties of fractional Sobolev spaces, we refer to \cite{Hitch}. Lemma \ref{sobolevembedding} below follows from Theorem 6.7. in \cite{Hitch}. The correct dependence upon $r$ is obtained by rescaling.  
\begin{theorem}\label{sobolevembedding}
Suppose $f\in H^s(B_r)$ for $s\in(0,1)$, $n\ge 2$ and let $\kappa^* = \frac{n}{n-2s}$. Then there exists a constant $C=C(n,s)$ such that 
\begin{equation*}
\left(\fint_{B_r}|f|^{2\kappa^*}dx\right)^{1/\kappa^*}\le Cr^{2s-n}\int_{B_r}\int_{B_r}\frac{|f(x)-f(y)|^2}{|x-y|^{n+2s}}dxdy + C\fint_{B_r}|f|^2dx. 
\end{equation*}
\end{theorem}
The next lemma is standard in the theory of parabolic pde. 
\begin{theorem}\label{parabsobolev}
Suppose $u\in L^2(t_1,t_2;H^s(B_r))$, $s\in(0,1)$ and let $\kappa^* = \frac{n}{n-2s}$.  
Then for any $\kappa\in[1,\kappa^*]$,
\begin{align*}
&\int_{t_1}^{t_2}\fint_{B_r}|f|^{2\kappa }dxdt\\
&\le Cr^{2s-n}\int_{t_1}^{t_2}[f(\cdot,t)]_{H^s(B_r)}^2dt
\times\left(\sup_{t_1<t<t_2}\fint_{B_r}|f|^{\frac{2\kappa^*(\kappa-1)}{\kappa^*-1}}dx\right)^{\frac{\kappa^*-1}{\kappa^*}}.\notag\\
&\le Cr^{-n}\int_{t_1}^{t_2}\|f(\cdot,t)\|_{L^2(B_r)}^2dt
\times\left(\sup_{t_1<t<t_2}\fint_{B_r}|f|^{\frac{2\kappa^*(\kappa-1)}{\kappa^*-1}}dx\right)^{\frac{\kappa^*-1}{\kappa^*}}.\notag
\end{align*}
\end{theorem}

\begin{proof}
By H\"older's inequality and Lemma \ref{sobolevembedding} we have 
\begin{align*}
&\int_{t_1}^{t_2}\fint_{B_r}|f|^{2\kappa}dxdt = \int_{t_1}^{t_2}\fint_{B_r}|f|^2|f|^{2(\kappa-1)}dxdt\\
&\le\int_{t_1}^{t_2}\left(\fint_{B_r}|f|^{2\kappa^* }dx\right)^{\frac{1}{\kappa^*}} \left(\fint_{B_r}|f|^{\frac{2\kappa^*(\kappa-1)}{\kappa^*-1}}dx\right)^{\frac{\kappa^*-1}{\kappa^*}}dt\\
&\le \left(Cr^{2s-n}\int_{t_1}^{t_2}\int_{B_r}\int_{B_r}\frac{|f(x)-f(y)|^2}{|x-y|^{n+sp}}dxdydt + C\int_{t_1}^{t_2}\fint_{B_r}|f|^2dxdt\right)\\
&\quad\times \left(\sup_{t_1<t<t_2}\fint_{B_r}|f|^{\frac{2\kappa^*(\kappa-1)}{\kappa^*-1}}dx\right)^{\frac{\kappa^*-1}{\kappa^*}}.
\end{align*}
\end{proof}

The following weighted Poincar\'e inequality is due to Dyda and Kassmann. See Corollary 6 in \cite{DydaKassmann}. The correct $r$-dependence is again obtained by rescaling. 
\begin{lemma}\label{DK}
Let $s\in (0,1)$ and let $\psi$ be a radially decreasing function on $B_r = B_r(x_0)$ of the form $\psi(x) = \Psi(|x-x_0|)$ such that $\psi\equiv 1$ in $B_{r/2}$. 
Then there exists a constant $C$ depending on $s,n$ such that for all $f\in L^2(B_r)$, 
\begin{equation*}
\int_{B_r}|f(x)-u_{\psi}|^2\psi(x)dx\le Cr^{2s}\int_{B_r}\int_{B_r}\frac{|f(x)-f(y)|^2}{|x-y|^{n+2s}}\min\{\psi(x),\psi(y)\}dxdy, 
\end{equation*}
where 
\[
u_{\psi} = \frac{\int_{B_r}u\psi dx}{\int_{B_r}\psi dx}.
\]
\end{lemma}

\subsection{Caccioppoli type inequalities}

In this section we derive inequalities of Caccioppoli type that play a key role in all subsequent estimates. 
The formal computations made in the proofs can be justified in view of Remarks \ref{rmk1} and \ref{rmk2}. 
For the following algebraic lemma we refer to \cite{FelKass} where it occurs as Lemma 3.3.
\begin{lemma}\label{lemma_alg}
Assume $q>1$, $a,b>0$ and $\alpha,\beta\ge 0$. Then there exists a constant $c_q\sim 1+q$ such that 
\begin{align}
(b-a)\left(\alpha^{q+1}a^{-q}-\beta^{q+1}b^{-q}\right) & \ge \frac{1}{q-1}\alpha\beta\left[\left(\frac{b}{\beta}\right)^{\frac{1-q}{2}} - \left(\frac{a}{\alpha}\right)^{\frac{1-q}{2}}\right]^2\tag{i}\\
& - c_q(\beta-\alpha)^2\left[\left(\frac{b}{\beta}\right)^{1-q} - \left(\frac{a}{\alpha}\right)^{1-q}\right]\notag
\end{align}
If $q\in(0,1)$, $a,b>0$ and $\alpha,\beta\ge 0$, there exist positive constants $c_{1,q}\sim \frac{q}{1-q}$ and $c_{2,q}\sim \frac{q}{1-q} + \frac{1}{q}$ such that 
\begin{align}
(b-a)(\alpha^2a^{-q}-\beta^2b^{-q})& \ge c_{1,q}\left(\beta b^{\frac{1-q}{2}} - \alpha a^{\frac{1-q}{2}}\right)^2\tag{ii}\\
& - c_{2,q}(\beta-\alpha)^2(b^{1-q} + a^{1-q}).\notag 
\end{align}
\end{lemma}

Lemma \ref{lemma_negativecacc} and \ref{lemma_positivecacc} below are, respectively, Caccioppoli inequalities for negative and small positive powers of supersolutions. They will be used in the proof of the weak Harnack inequality. In the case of supersolutions that are nonnegative in all space, they occur implicitly in \cite{FelKass}. We here allow the supersolutions to go below zero and thus need to additionally take into account the contribution of their negative parts.

\begin{lemma}\label{lemma_negativecacc}
Let $x_0\in\R^n$ and for any $\rho>0$, let $B_\rho = B_\rho(x_0)$. Let $0<r<R$ and let $p>0$. Suppose $u$ is a supersolution to \eqref{maineq} such that 
$$u\ge 0\text{ in }B_R\times(\tau_1-\lag,\tau_2),\quad (\tau_1-\lag,\tau_2)\subset (0,T).$$ Then for any $d>0$ and $\tilde u = u+d$, there exists a constant $C = C(n,s,\Lambda,p)$ that behaves like $C_0(n,s,\Lambda)(1+p^2)$, such that 
\begin{align}\label{cclemmanegative}
& \int_{\tau_1-\lag}^{\tau_2}\int_{B_r}\int_{B_r}\psi(x)\psi(y)\left[\left(\frac{\tilde u(x,t)}{\psi(x)}\right)^{-\frac{p}{2}} - \left(\frac{\tilde u(y,t)}{\psi(y)}\right)^{-\frac{p}{2}}\right]^2\eta(t)d\mu dt\\
& + \sup_{\tau_1<t<\tau_2}\int_{B_r}\psi^{p+2}(x)\tilde u^{-p}(x,t)dx\notag\\
& \le C\int_{\tau_1-\lag}^{\tau_2}\int_{B_r}\int_{B_r}(\psi(x)-\psi(y))^2\left[\left(\frac{\tilde u(x,t)}{\psi(x)}\right)^{-p} + \left(\frac{\tilde u(y,t)}{\psi(y)}\right)^{-p}\right]\eta(t)d\mu dt\notag\\
& + C\sup_{x\in\supp\psi}\int_{\R^n\setminus B_r}\frac{dy}{|x-y|^{n+2s}} \int_{\tau_1-\lag}^{\tau_2}\int_{B_r} \tilde u^{-p}(x,t)\psi^{p+2}(x,t)\eta(t)dxdt \notag\\
& + \frac{C}{d}\sup_{\stackrel{\tau_1-\lag<t<\tau_2}{x\in\supp\psi}}\int_{\R^n\setminus B_R}\frac{u_-(y,t)dy}{|x-y|^{n+2s}}\int_{\tau_1-\lag}^{\tau_2}\int_{B_r} \tilde u^{-p}(x,t)\psi^{p+2}(x,t)\eta(t)dxdt\notag\\
& + C\int_{\tau_1-\lag}^{\tau_2}\int_{B_r}\psi^{p+2}(x)\tilde u^{-p}(x,t)\partial_t\eta(t)dxdt, \notag
\end{align}
for all nonnegative $\psi\in C_0^\infty(B_r)$ and nonnegative $\eta\in C^{\infty}(\R)$ such that $\eta(t)\equiv 0 $ if $t\le \tau_1-\lag$ and $\eta\equiv 1$ if $t\ge \tau_2$. 
\end{lemma}

\begin{proof}
Let $\tilde u = u+d$ and let $\psi\in C_c^\infty(B_r)$. Let $t_1=\tau_1-\lag$, let $t_2\in(\tau_1,\tau_2)$ and let $\eta\in C^\infty(t_1,t_2)$ satisfy $\eta(t_1)=0$ and $\eta(t)=1$ for all $t\ge t_2$. Define, for $q>1$, 
\[
v(x,t) = \tilde u^{\frac{1-q}{2}}(x,t),\quad \phi(x,t) = \tilde u^{-q}\psi^{q+1}\eta(t). 
\]
Since $\tilde u$ is a supersolution we obtain 
\begin{align}\label{Is}
0 & \le \int_{t_1}^{t_2}\int_{B_r} \partial_t\tilde u(x,t)\phi(x,t)dxdt\\
& + \int_{t_1}^{t_2}\int_{B_r}\int_{B_r}(\tilde u(x,t) - \tilde u(y,t))(\phi(x,t)-\phi(y,t))d\mu(x,y,t)dt\notag\\
& + 2\int_{t_1}^{t_2}\int_{\R^n\setminus B_r}\int_{B_r}(\tilde u(x,t) - \tilde u(y,t))\phi(x,t)d\mu(x,y,t)dt\notag\\
&=-\frac{1}{q-1}\left[\int_{B_r}\psi^{q+1}(x)\eta(t)v^2(x,t)dx\right]_{t_1}^{t_2}\notag\\
&+ \frac{1}{q-1}\int_{t_1}^{t_2}\int_{B_r}\psi^{q+1}(x)v^2(x,t)\partial_t\eta(t)dxdt\notag\\
& + \int_{t_1}^{t_2}\int_{B_r}\int_{B_r}(\tilde u(x,t) - \tilde u(y,t))\left(\frac{\psi^{q+1}(x)}{\tilde u^{q}(x,t)} - \frac{\psi^{q+1}(y)}{\tilde u^{q}(y,t)}\right)\eta(t)d\mu(x,y,t)dt\notag\\
& + 2\int_{t_1}^{t_2}\int_{\R^n\setminus B_r}\int_{B_r}(\tilde u(x,t) - \tilde u(y,t))\psi^{q+1}(x,t)\tilde u^{-q}(x,t)\eta(t)d\mu(x,y,t)dt\notag\\
& = I_0 + I_1 + I_2 + I_3.\notag
\end{align}
Since $u\ge 0$ in $B_R\times(t_1,t_2)$ we have, using that $d\le \tilde u$ in $B_R\times(t_1,t_2)$, 
\begin{align}\label{I3}
I_3 & \le 2\int_{t_1}^{t_2}\int_{\R^n\setminus B_r}\int_{B_r}\tilde v^2(x,t)\psi^{q+1}(x,t)\eta(t)d\mu(x,y,t)dt\\
& + \frac{2}{d}\int_{t_1}^{t_2}\int_{\R^n\setminus B_R}\int_{B_r} u_-(y,t)v^2(x,t)\psi^{q+1}(x,t)\eta(t)d\mu(x,y,t)dt\notag\\
& \le 2\Lambda\sup_{x\in\supp\psi}\int_{\R^n\setminus B_r}\frac{dy}{|x-y|^{n+2s}} \int_{t_1}^{t_2}\int_{B_r} v^2(x,t)\psi^{q+1}(x,t)\eta(t)dxdt \notag\\
& + \frac{2\Lambda}{d}\sup_{\stackrel{t_1<t<t_2}{x\in\supp\psi}}\int_{\R^n\setminus B_R}\frac{u_-(y,t)dy}{|x-y|^{n+2s}}\int_{t_1}^{t_2}\int_{B_r} v^2(x,t)\psi^{q+1}(x,t)\eta(t)dxdt.\notag
\end{align}
For $I_2$, we use Lemma \ref{lemma_alg} to estimate 
\begin{align}\label{I2}
&-I_2  \ge \int_{t_1}^{t_2}\int_{B_r}\int_{B_r}\frac{\psi(x)\psi(y)}{q-1}\left[\left(\frac{\tilde u(x,t)}{\psi(x)}\right)^{\frac{1-q}{2}} - \left(\frac{\tilde u(y,t)}{\psi(y)}\right)^{\frac{1-q}{2}}\right]^2\eta(t)d\mu dt\\
& - c_q\int_{t_1}^{t_2}\int_{B_r}\int_{B_r}(\psi(x)-\psi(y))^2\left[\left(\frac{\tilde u(x,t)}{\psi(x)}\right)^{1-q} + \left(\frac{\tilde u(y,t)}{\psi(y)}\right)^{1-q}\right]\eta(t)d\mu dt\notag
\end{align}
We now choose $t_2$ such that 
\begin{align}\label{I0}
-I_0 = \frac{1}{q-1}\int_{B_r}\psi^{q+1}(x)v^2(x,t_2)dx = \sup_{\tau_1<t<\tau_2}\frac{1}{q-1}\int_{B_r}\psi^{q+1}(x)v^2(x,t)dx. 
\end{align}
Using \eqref{I3}, \eqref{I2} and \eqref{I0} in \eqref{Is}, we obtain 
\begin{align}\label{ccnegative}
& \int_{t_1}^{t_2}\int_{B_r}\int_{B_r}\frac{\psi(x)\psi(y)}{q-1}\left[\left(\frac{\tilde u(x,t)}{\psi(x)}\right)^{\frac{1-q}{2}} - \left(\frac{\tilde u(y,t)}{\psi(y)}\right)^{\frac{1-q}{2}}\right]^2\eta(t)d\mu dt\\
& + \sup_{\tau_1<t<\tau_2}\frac{1}{q-1}\int_{B_r}\psi^{q+1}(x)v^2(x,t)dx\notag\\
& \le c_q\int_{t_1}^{t_2}\int_{B_r}\int_{B_r}(\psi(x)-\psi(y))^2\left[\left(\frac{\tilde u(x,t)}{\psi(x)}\right)^{1-q} + \left(\frac{\tilde u(y,t)}{\psi(y)}\right)^{1-q}\right]\eta(t)d\mu dt\notag\\
& + 2\Lambda\sup_{x\in\supp\psi}\int_{\R^n\setminus B_r}\frac{dy}{|x-y|^{n+2s}} \int_{t_1}^{t_2}\int_{B_r} v^2(x,t)\psi^{q+1}(x,t)\eta(t)dxdt \notag\\
& + \frac{2\Lambda}{d}\sup_{\stackrel{t_1<t<t_2}{x\in\supp\psi}}\int_{\R^n\setminus B_R}\frac{u_-(y,t)dy}{|x-y|^{n+2s}}\int_{t_1}^{t_2}\int_{B_r} v^2(x,t)\psi^{q+1}(x,t)\eta(t)dxdt\notag\\
& + \frac{1}{q-1}\int_{t_1}^{t_2}\int_{B_r}\psi^{q+1}(x)v^2(x,t)\partial_t\eta(t)dxdt.\notag
\end{align}
If we choose $t_2 = \tau_2$, we see that \eqref{ccnegative} holds with $\frac{1}{q-1}\int_{B_r}\psi^{q+1}(x)v^2(x,\tau_2)dx$ in place of $\sup_{\tau_1<t<\tau_2}\frac{1}{q-1}\int_{B_r}\psi^{q+1}(x)v^2(x,t)dx$. Let with $p=(q-1)/2$. Then $c_q\sim 1+p$ by Lemma \ref{lemma_alg} (i). This completes the proof of \eqref{cclemmanegative}. 

\end{proof}

\begin{lemma}\label{lemma_positivecacc}
Let $x_0\in\R^n$ and for any $\rho>0$, let $B_\rho = B_\rho(x_0)$. Let $0<r<R$ and $p\in(p_1,p_2)\subset(0,1)$. Suppose that $u$ is a supersolution to \eqref{maineq} such that 
$$u\ge 0\text{ in }B_R\times(\tau_1,\tau_2 + \lag),\quad (\tau_1,\tau_2 + \lag)\subset(0,T).$$ Then for any $d>0$ and $\tilde u = u+d$, there exists a constant $C = C(n,s,\Lambda,p_1,p_2)$ such that 
\begin{align}\label{cclemmapositive}
& \int_{\tau_1}^{\tau_2+\lag}\int_{B_r}\int_{B_r}\left[\tilde u(x,t)^{\frac{p}{2}}\psi(x) - \tilde u(y,t)^{\frac{p}{2}}\psi(y)\right]^2\eta(t)d\mu dt\\
& + \sup_{\tau_1<t<\tau_2}\int_{B_r}\psi^{2}(x)\tilde u^{p}(x,t)dx\notag\\
& \le C\int_{\tau_1}^{\tau_2+\lag}\int_{B_r}\int_{B_r}(\psi(x)-\psi(y))^2\left(\tilde u(x,t)^p + \tilde u(y,t)^p\right)\eta(t)d\mu dt\notag\\
& + C\sup_{x\in\supp\psi}\int_{\R^n\setminus B_r}\frac{dy}{|x-y|^{n+2s}} \int_{\tau_1}^{\tau_2+\lag}\int_{B_r} \tilde u^{p}(x,t)\psi^{2}(x,t)\eta(t)dxdt \notag\\
& + \frac{C}{d}\sup_{\substack{\tau_1<t<\tau_2+\lag\\x\in\supp\psi}}\int_{\R^n\setminus B_R}\frac{u_-(y,t)dy}{|x-y|^{n+2s}}\int_{\tau_1}^{\tau_2+\lag}\int_{B_r} \tilde u^{p}(x,t)\psi^{2}(x,t)\eta(t)dxdt\notag\\
& + C\int_{\tau_1}^{\tau_2+\lag}\int_{B_r}\psi^{2}(x)\tilde u^{p}(x,t)\partial_t\eta(t)dxdt, \notag
\end{align}
for all nonnegative $\psi\in C_0^\infty(B_r)$ and nonnegative $\eta\in C^{\infty}(\R)$ such that $\eta(t)\equiv 1 $ if $\tau_1\le t\le \tau_2$ and $\eta\equiv 0$ if $t\ge \tau_2+\lag$. 
\end{lemma}

\begin{proof}
Let $\tilde u = u+d$ and let $\psi\in C_c^\infty(B_r)$. Let $t_1\in(\tau_1,\tau_2)$, let $t_2 = \tau_2 + \lag$ and let $\eta\in C^\infty(t_1,t_2)$ satisfy $\eta(t_2)=0$ and $\eta(t)=1$ for all $t\le t_1$. Define, for $q\in(0,1)$, 
\[
v(x,t) = \tilde u^{\frac{1-q}{2}}(x,t),\quad \phi(x,t) = \tilde u^{-q}\psi^{2}\eta(t). 
\]
Since $\tilde u$ is a supersolution we have
\begin{align}\label{Is'}
0 & \le \int_{t_1}^{t_2}\int_{B_r} \partial_t\tilde u(x,t)\phi(x,t)dxdt\\
& + \int_{t_1}^{t_2}\int_{B_r}\int_{B_r}(\tilde u(x,t) - \tilde u(y,t))(\phi(x,t)-\phi(y,t))d\mu(x,y,t)dt\notag\\
& + 2\int_{t_1}^{t_2}\int_{\R^n\setminus B_r}\int_{B_r}(\tilde u(x,t) - \tilde u(y,t))\phi(x,t)d\mu(x,y,t)dt\notag\\
&=-\frac{1}{q-1}\left[\int_{B_r}\psi^{2}(x)\eta(t)v^2(x,t)dx\right]_{t_1}^{t_2}\notag\\
&+ \frac{1}{q-1}\int_{t_1}^{t_2}\int_{B_r}\psi^{2}(x)v^2(x,t)\partial_t\eta(t)dxdt\notag\\
& + \int_{t_1}^{t_2}\int_{B_r}\int_{B_r}(\tilde u(x,t) - \tilde u(y,t))\left(\frac{\psi^{2}(x)}{\tilde u^{q}(x,t)} - \frac{\psi^{2}(y)}{\tilde u^{q}(y,t)}\right)\eta(t)d\mu(x,y,t)dt\notag\\
& + 2\int_{t_1}^{t_2}\int_{\R^n\setminus B_r}\int_{B_r}(\tilde u(x,t) - \tilde u(y,t))\psi^{2}(x,t)\tilde u^{-q}(x,t)\eta(t)d\mu(x,y,t)dt\notag\\
& = I_0 + I_1 + I_2 + I_3.\notag
\end{align}
Since $u\ge 0$ in $B_R\times(t_1,t_2)$ we have, using that $d\le \tilde u$ in $B_R\times(t_1,t_2)$, 
\begin{align}\label{I3'}
I_3 & \le 2\int_{t_1}^{t_2}\int_{\R^n\setminus B_r}\int_{B_r} v^2(x,t)\psi^{2}(x,t)\eta(t)d\mu(x,y,t)dt\\
& + \frac{2}{d}\int_{t_1}^{t_2}\int_{\R^n\setminus B_R}\int_{B_r} u_-(y,t)v^2(x,t)\psi^{2}(x,t)\eta(t)d\mu(x,y,t)dt\notag\\
& \le 2\Lambda\sup_{x\in\supp\psi}\int_{\R^n\setminus B_r}\frac{dy}{|x-y|^{n+2s}} \int_{t_1}^{t_2}\int_{B_r} v^2(x,t)\psi^{2}(x,t)\eta(t)dxdt \notag\\
& + \frac{2\Lambda}{d}\sup_{\stackrel{t_1<t<t_2}{x\in\supp\psi}}\int_{\R^n\setminus B_R}\frac{u_-(y,t)dy}{|x-y|^{n+2s}}\int_{t_1}^{t_2}\int_{B_r} v^2(x,t)\psi^{2}(x,t)\eta(t)dxdt.\notag
\end{align}
For $I_2$, we use Lemma \ref{lemma_alg} to estimate 
\begin{align}\label{I2'}
-I_2 & \ge c_{1,q}\int_{t_1}^{t_2}\int_{B_r}\int_{B_r}\left(\psi(x) v(x,t) - \psi(y) v(y,t)\right)^2\eta(t)d\mu dt\\
& - c_{2,q}\int_{t_1}^{t_2}\int_{B_r}\int_{B_r}(\psi(x)-\psi(y))^2(v^2(x,t) + v^2(y,t))\eta(t)d\mu dt.\notag
\end{align}
We now choose $t_1$ such that 
\begin{align}\label{I0'}
-I_0 = -\frac{1}{q-1}\int_{B_r}\psi^{2}(x)v^2(x,t_1)dx = \sup_{\tau_1<t<\tau_2}\frac{1}{1-q}\int_{B_r}\psi^{2}(x)v^2(x,t)dx. 
\end{align}
Using \eqref{I3'}, \eqref{I2'} and \eqref{I0'} in \eqref{Is'}, we obtain 
\begin{align}\label{ccpositive}
& c_{1,q}\int_{t_1}^{t_2}\int_{B_r}\int_{B_r}\left(\psi(x) v(x,t) - \psi(y) v(y,t)\right)^2\eta(t)d\mu dt\\
& + \sup_{\tau_1<t<\tau_2}\frac{1}{1-q}\int_{B_r}\psi^{2}(x)v^2(x,t)dx\notag\\
& \le c_{2,q}\int_{t_1}^{t_2}\int_{B_r}\int_{B_r}(\psi(x)-\psi(y))^2(v^2(x,t) + v^2(y,t))\eta(t)d\mu dt\notag\\
& + \le 2\Lambda\sup_{x\in\supp\psi}\int_{\R^n\setminus B_r}\frac{dy}{|x-y|^{n+2s}} \int_{t_1}^{t_2}\int_{B_r} v^2(x,t)\psi^{2}(x,t)\eta(t)dxdt \notag\\
& + \frac{2\Lambda}{d}\sup_{\stackrel{t_1<t<t_2}{x\in\supp\psi}}\int_{\R^n\setminus B_R}\frac{u_-(y,t)dy}{|x-y|^{n+2s}}\int_{t_1}^{t_2}\int_{B_r} v^2(x,t)\psi^{2}(x,t)\eta(t)dxdt \notag\\
& - \frac{1}{q-1}\int_{t_1}^{t_2}\int_{B_r}\psi^{2}(x)v^2(x,t)\partial_t\eta(t)dxdt\notag
\end{align}
If we choose $t_1 = \tau_1$, we see that \eqref{cclemmapositive} holds with $\frac{1}{1-q}\int_{B_r}\psi^{2}(x)v^2(x,\tau_1)dx$ in place of $\sup_{\tau_1<t<\tau_2}\frac{1}{1-q}\int_{B_r}\psi^{2}(x)v^2(x,t)dx$. This proves \eqref{cclemmapositive} with $p=1-q$. If $p\in(p_1,p_2)$, the constants $c_{1,q},c_{2,q}$ from Lemma \ref{lemma_alg} and $1/(1-q)$ can be bounded in terms of $p_1,p_2$ only. 

\end{proof}

Finally we need a Caccioppoli inequality for subsolutions. This is based on Theorem 1.4. in \cite{DKPpmin}. 

\begin{lemma} \label{lemma_subsolcacc}
Let $x_0\in\R^n$ and for any $\rho>0$, let $B_\rho = B_\rho(x_0)$. Suppose that $u$ is a subsolution to \eqref{maineq} and let $0<\tau_1<\tau_2$ and $\lag>0$ satisfy $(\tau_1-\lag,\tau_2)\subset(0,T)$. 
Then there exists a constant $C = C(n,s,\Lambda)$ such that
\begin{align}\label{eq_subsolcacc}
&\int_{\tau_1-\lag}^{\tau_2}\int_{B_r}\int_{B_r}|u(x,t)\psi(x)-u(y,t)\psi(y)|^2\eta^2(t)d\mu dt\\
& + \frac{1}{2}\sup_{\tau_1<t<\tau_2}\int_{B_r}u^2(x,t)\psi^2(x)dx \notag\\
& \le C\int_{\tau_1-\lag}^{\tau_2}\int_{B_r}\int_{B_r}\max\{u^2(x,t),u^2(y,t)\}|\psi(x)-\psi(y)|^2\eta^2(t)d\mu dt\notag\\
& + C\sup_{\substack{\tau_1-\lag<t<\tau_2\\x\in\supp\psi}}\int_{\R^n\setminus B_r}\frac{u_+(y,t)dy}{|x-y|^{n+2s}}\int_{\tau_1-\lag}^{\tau_2}\int_{B_r}u(x,t)\psi^2(x)\eta^2(t)dxdt\notag\\
& + \frac{1}{2}\int_{\tau_1-\lag}^{\tau_2}\int_{B_r}u^2(x,t)\psi^2(x)\partial_t\eta^2(t)dxdt.\notag 
\end{align}
    for all nonnegative $\psi\in C_0^\infty(B_r)$ and nonnegative $\eta\in C^{\infty}(\R)$ such that $\eta(t)\equiv 0 $ if $t\le \tau_1-\lag$ and $\eta\equiv 1$ if $t\ge \tau_1$. 
    
	\end{lemma}

\begin{proof}
Let $t_1 = \tau_1-\lag$ and let $t_2\in(t_1,\tau_2]$. Using $\phi(x,t) = u(x,t)\psi^2(x)\eta^2(t)$ as a test function in \eqref{weak2}, appealing to Remark \ref{rmk2}, we get 
\begin{align}\label{subcaccIs}
0 & \ge \int_{t_1}^{t_2}\int_{B_r}\int_{B_r}(u(x,t)-u(y,t))(u(x,t)\psi^2(x)-u(y,t)\psi^2(y))\eta^2(t)d\mu dt\\
& + 2\int_{t_1}^{t_2}\int_{\R^n\setminus B_r}\int_{B_r}
(u(x,t)-u(y,t))u(x)\psi^2(x)\eta^2 d\mu dt\notag\\
& + \int_{t_1}^{t_2}\int_{B_r}u(x,t)\eta^2(t)\psi^2(x)\partial_t u(x,t)dxdt\notag\\
& = I_1 + I_2 + I_3.\notag 
\end{align}
Using the assumptions on $\eta$ and integrating by parts, we find 
\begin{align}\label{subcaccI3}
I_3 & = \int_{t_1}^{t_2}\int_{B_r}\eta^2(t)\psi^2(x)\partial_t \frac{u^2(x,t)}{2}dxdt\\
& = \int_{B_r}\frac{u^2(x,t_2)}{2}\psi^2(x)dx - 
\int_{t_1}^{t_2}\int_{B_r}\frac{u^2(x,t)}{2}\psi^2(x)\partial_t\eta^2(t)dxdt. \notag
\end{align}
Turning then to $I_2$ we have 
\begin{align}\label{subcaccI2}
I_2 & \ge -2\int_{t_1}^{t_2}\int_{\R^n\setminus B_r}\int_{B_r}
u(y,t)u(x,t)\psi^2(x)\eta^2 d\mu dt\\
& \ge -2\Lambda\int_{t_1}^{t_2}\int_{\R^n\setminus B_r}\int_{B_r}
\frac{u_+(y,t)}{|x-y|^{n+2s}}u(x,t)\psi^2(x)\eta^2 dxdy dt\notag\\
& \ge -2\Lambda\sup_{\substack{t_1<t<t_2\\x\in\supp\psi}}
\int_{\R^n\setminus B_r}\frac{u_+(y,t)dy}{|x-y|^{n+2s}}\int_{t_1}^{t_2}\int_{B_r}u(x,t)\psi^2(x)\eta^2(t)dxdt.\notag
\end{align}
For the estimation of $I_1$ we refer to the proof of Theorem 1.4 in \cite{DKPpmin}, where it is shown that 
\begin{align}\label{subcaccI1}
I_1 & \ge \frac12\int_{t_1}^{t_2}\int_{B_r}\int_{B_r}|u(x,t)\psi(x)-u(y,t)\psi(y)|^2\eta^2(t)d\mu dt\\
& - C\int_{t_1}^{t_2}\int_{B_r}\int_{B_r}\max\{u^2(x,t),u^2(y,t)\}|\psi(x)-\psi(y)|^2\eta^2(t)d\mu dt.\notag
\end{align}
If we use the estimates \eqref{subcaccI1}, \eqref{subcaccI2} and \eqref{subcaccI3} for $I_1$, $I_2$ and $I_3$ in \eqref{subcaccIs}, and choose $t_2= \tau_2$, we arrive at the desired conclusion save for the term 
\[
\frac{1}{2}\sup_{\tau_1<t<\tau_2}\int_{B_r}u^2(x,t)\psi^2(x)dx. 
\]
If we choose $t_2$ such that 
\[
\frac{1}{2}\int_{B_r}u^2(x,t_2)\psi^2(x)dx = \frac{1}{2}\sup_{\tau_1<t<\tau_2}\int_{B_r}u^2(x,t)\psi^2(x)dx, 
\]
we obtain an estimate for $\frac{1}{2}\sup_{\tau_1<t<\tau_2}\int_{B_r}u^2(x,t)\psi^2(x)dx$ in terms of the right hand side of \eqref{eq_subsolcacc}, with $t_2$ in place of $\tau_2$. 
This completes the proof. 
\end{proof}

\subsection{Estimation of Tails}

The remainder of this section is devoted to estimates of the tails in Definition \ref{def_tail}. We basically need two things here: 1. An estimate of the supremum version of the tail \eqref{thesupTail} in terms of "weaker" tail in \eqref{theTail}. 2. An estimate of $\tail(u_+;\cdots)$ in terms of $\tail(u_-,\cdots)$ and the local supremum of $u$. Point 2. can not be done for the supremum version of the tail directly, which is why point 1. is so important. Here we use an important tool from \cite{BV}. 

\begin{lemma}\label{lemma_eig}
Let $\Phi(x)$ be defined by 
\begin{equation*}
	\Phi(x) = 
    \left\{\begin{array}{l}
		1\text{ if }|x|<1,\\
        \frac{1}{(1+(|x|^2-1)^4)^{(n+2s)/8}}\text{ if }|x|\ge 1,
	\end{array}\right. 
\end{equation*}
and let $\Phi_r(x) = r^{-n}\Phi(x/r)$.
Then there exist constants $c_1\ge 1$ and $c_2\ge 1$, depending only on $n,s,\Lambda$, such that 
 
\begin{align}\label{LPhi_r}
&c_1^{-1}r^{-2s}\Phi_r(x) \le |L\Phi_r(x)| \le c_1r^{-2s}\Phi_r(x),\tag{i}\\
&c_2^{-1}\frac{r^{2s}}{|x|^{n+2s}}\le \Phi_r(x) \le c_2\frac{r^{2s}}{|x|^{n+2s}},\quad\text{for all }|x|\ge r. \label{Phi_r2}\tag{ii}
\end{align}
\end{lemma}
\begin{proof}
The estimate \eqref{LPhi_r} is proved in \cite{BV}, in the case that $L=(-\Delta)^s$ and $r=1$. However, the proof can be easily adapted to symmetric kernels $K$ satisfying \eqref{ellipticity}. The constant $c_1$ will depend only on the ellipticity constant $\Lambda$. This establishes \eqref{LPhi_r} for $r=1$. 
For the rescaled function $\Phi_r$ we have, setting $z=x/r$ and $\eta = y/r$ ,
\begin{align*}
& L\Phi_r(x) = r^{-n}\int_{\R^n}K(x,y,t)(\Phi(x/r)-\Phi(y/r))dy\\
& = r^{-n-2s}\int_{\R^n}K(rz,r\eta,t)r^{n+2s}(\Phi(z)-\Phi(\eta))d\eta=: r^{-n-2s}(L_r\Phi)(x/r).
\end{align*}
The operator $L_r$, defined through the kernel 
\[K_r(x,y,t) = K(rz,r\eta,t)r^{n+2s},\] has the same ellipticity constants as $L$. Hence \eqref{LPhi_r} follows. It is easy to check \eqref{Phi_r2}  from the definition. 
\end{proof}

\begin{lemma}\label{tail_inf_le_tail}
Let $x_0\in \R^n$, $r>0$ and let $t_1,t_2$ satisfy $r^{2s}<t_1<T-r^{2s}$, $t_2=t_1+r^{2s}$.  
Suppose that $u$ is a weak subsolution to \eqref{maineq} that is nonnegative in $B_r(x_0)\times(t_1,t_2)$. Then for any $0<\e <r^{-2s}t_1$,  
\begin{align*}
\tail_\infty(u_+;x_0,r,t_1,t_2) & \le  C\e^{-1}\tail(u_+;x_0,r,t_1-\e r^{2s},t_2)\\
&+ C\e^{-1}\fint_{t_1-\e r^{2s}}^{t_2}\fint_{B_r(x_0)}u_+dxdt. 
\end{align*}

\end{lemma}

\begin{proof}
It may be assumed that $x_0=0$. Let $\delta>0$ and let $\tau=\tau_\delta\in(t_1,t_2)$ satisfy 
\begin{equation}\label{supminuse}
r^{2s}\int_{\R^n\setminus B_r}\frac{u_+(x,\tau)}{|x|^{n+2s}}dx\ge \tail_\infty(u_+;0,r,t_1,t_2)-\delta.
\end{equation}
Let $\Phi_r$ be the function in Lemma \ref{lemma_eig}. Let further $\eta\in C^\infty(\R)$ be a function satisfying $\eta\equiv 1$ in $[\tau,t_2]$, $\eta(t) = 0$ for $t\le t_1-\e r^{2s}$ and $|\eta'|\le C\e^{-1}r^{-2s}$. 
We recall from Lemma \ref{lemma_subsol} that $u_+$ is a weak subsolution and use $\phi = \Phi_r\eta$ as test function: 
\begin{align*}
&\int_{t_1-\e r^{2s}}^{\tau}\int_{\R^n}\Phi_r\eta\partial_tu_+dxdt\\
&+\int_{t_1-\e r^{2s}}^{\tau}\int_{\R^n}\int_{\R^n}(u_+(x,t)-u_+(y,t))(\Phi_r(x)-\Phi_r(y))\eta(t)d\mu(x,y,t)dt\le 0. 
\end{align*}
We then integrate by parts, to find
\begin{align*}
&\int_{\R^n}u_+(x,\tau)\Phi_r(x)dx\le \int_{t_1-\e r^{2s}}^{\tau}\int_{\R^n}u_+\Phi_r\partial_t\eta dxdt\\
&- 2\int_{t_1-\e r^{2s}}^{\tau}\int_{\R^n}u_+(x,t)\L\Phi_r(x)\eta(t)dxdt.
\end{align*}
Using Lemma \ref{lemma_eig} and the definition of $\eta$ yields
\begin{align*}
& \int_{\R^n}u_+(x,\tau)\Phi_r(x)dx\label{eq112}\\
&\le C\e^{-1}r^{-2s}\int_{t_1-\e r^{2s}}^{\tau}\int_{\R^n}u_+\Phi_rdxdt + 
C\int_{t_1-\e r^{2s}}^{\tau}\int_{\R^n}u_+r^{-2s}\Phi_r(x)dxdt\\\notag
&\le C\e^{-1}\fint_{t_1-\e r^{2s}}^{\tau}\fint_{B_r(x_0)}u_+dxdt + 
C\e^{-1}\int_{t_1-\e r^{2s}}^{\tau}\int_{\R^n\setminus B_r(x_0)}\frac{u_+}{|x-x_0|^{n+2s}}dxdt\\\notag
&\le C\e^{-1}\fint_{t_1-\e r^{2s}}^{t_2}\fint_{B_r(x_0)}u_+dxdt 
+ C\e^{-1}\tail(u_+;x_0,r,t_1-\e r^{2s},t_2), \notag
\end{align*}
where we used that $t_2-(t_1-\e r^{2s})\approx r^{2s}$. It is a consequence of the definition of $\Phi_r$ that 
\[
r^{2s}\int_{\R^n\setminus B_r}\frac{u_+(x,\tau)}{|x|^{n+2s}}dx\le C\int_{\R^n}u_+(x,\tau)\Phi_r(x)dx. 
\]
The lemma now follows from \eqref{supminuse} 
since $\delta$ is arbitrary. 
\end{proof}

\begin{corollary}\label{tail_inf_le_tail_minus}
Suppose $u$ is a weak supersolution to \eqref{maineq}. Let $x_0\in\R^n$ and $r>0$. Then for any $r^{2s}<t_1<T-r^{2s}$, $t_2=t_1+r^{2s}$ and any $0<\e<t_1r^{-2s}$,  
\begin{align*}
\tail_\infty(u_-;x_0,r,t_1,t_2) & \le  C\e^{-1}\tail(u_-;x_0,r,t_1-\e r^{2s},t_2)\\
&+ C\e^{-1}\fint_{t_1-\e r^{2s}}^{t_2}\fint_{B_r(x_0)}u_-dxdt. 
\end{align*}

\end{corollary}

\begin{proof}
Since $u$ is a supersolution, $v=-u$ is a subsolution. Thus, by Lemma \ref{lemma_subsol}, $u_- = v_+$ is a subsolution and the result follows from Lemma \ref{tail_inf_le_tail}.  
\end{proof}

\begin{lemma}\label{lemma_tail+-}
Let $u$ be a weak solution to \eqref{maineq}. For $0<r<R/2$, suppose that $$u\ge 0\text{ in }B_R(x_0)\times(t_1,t_2),$$ where $0<t_1<T-r^{2s}$ and $t_2=t_1+r^{2s}$. 
Then 
\begin{equation}
\tail(u_+;x_0,r,t_1,t_2)\le C\sup_{B_r(x_0)\times(t_1,t_2)}u + C\left(\frac{r}{R}\right)^{2s}\tail(u_-;x_0,R,t_1,t_2). 
\end{equation}

\end{lemma}
\begin{proof}
Let $\psi\in C_c^\infty(B_{3r/4})$ satisfy $\psi\equiv 1$ in $B_{r/2}$, $0\le \psi\le 1$ and $|\nabla\psi|\le C/r$. 
Let $k= \sup_{B_r\times(t_1,t_2)}u$. We test the equation \eqref{maineq} with $\phi = (u-2k)\psi^2$: 
\begin{align}\label{123}
0 & = \int_{t_1}^{t_2}\int_{B_r}\partial_tu\phi dxdt\\
&+ \int_{t_1}^{t_2}\int_{B_r}\int_{B_r}(u(x,t)-u(y,t))(\phi(x,t)-\phi(y,t))dxdydt\notag\\
&+ 2\int_{t_1}^{t_2}\int_{\R^n\setminus B_r}\int_{B_r}(u(x,t)-u(y,t))(u-2k)\psi^2 dxdydt\notag\\
&=I_{1} + I_{2} + I_{3}.\notag 
\end{align}
Integrating by parts in $I_1$, we immediately obtain 
\begin{equation*}
I_1 = \int_{t_1}^{t_2}\frac{\partial_t(u-2k)^2}{2}\psi^2dxdt = \frac12\int_{B_r}((u(x,t_2)-2k)^2-(u(x,t_1)-2k)^2)\psi^2(x)dx.
\end{equation*}
Hence 
\begin{equation}\label{1}
|I_1|\le Cr^nk^2. 
\end{equation}
We next estimate the integrand of $I_{2}$ under the assumption that $\psi(x)>\psi(y)$. Letting $w = (u-2k)$, we have 
\begin{align*}
&(w(x,t)-w(y,t))(w(x,t)\psi^2(x)-w(y,t)\psi^2(y))\\
& = (w(x,t)-w(y,t))^2\psi^2(x) - (w(x,t)-w(y,t))w(y,t)(\psi^2(x)-\psi^2(y))\\
&\ge (w(x,t)-w(y,t))^2\psi^2(x) - |w(x,t)-w(y,t)||w(y,t)|\psi(x)|\psi(x)-\psi(y)|\\
&\ge -k^2|\psi(x)-\psi(y)|^2, 
\end{align*}
where we used Young's inequality and the fact that $|w|\le k$ in $B_r\times(t_1,t_2)$. The same estimate is clearly valid if $\psi(y)\ge \psi(x)$ as can be seen by interchanging the roles of $x$ and $y$. We thus obtain 
\begin{align}\label{2}
I_2& \ge -Ck^2\int_{t_1}^{t_2}\int_{B_r}\int_{B_r}|\psi(x)-\psi(y)|^2d\mu dt\\
&\ge -Ck^2r^{-2}\int_{t_1}^{t_2}\int_{B_r}\int_{B_r}|x-y|^{2-n-2s}dxdydt \ge 
-Ck^2r^n.\notag
\end{align}
Using $\psi\equiv 1 $ in $B_{r/2}$, we find the following lower bound for $I_3$: 
\begin{align}\label{3}
I_3 & \ge \int_{t_1}^{t_2}\int_{\R^n\setminus B_r}\int_{B_{r/2}}(u(y,t)-k)_+kd\mu dt\\
&-2k\int_{t_1}^{t_2}\int_{\R^n\setminus B_r}\int_{B_r}(u(x,t)-u(y,t))_+\chi_{(u(y,t)<k)}\psi^2(x)d\mu dt\notag\\
&=I_{31}-I_{32}. \notag
\end{align}
Since $|x-y|\le |x| + |y|\le 2|y|$ whenever $x\in B_r$ and $y\in\R^n\setminus B_r$,
\begin{align}\label{31}
I_{31}&\ge C_0kr^n\tail(u_+;x_0,r,t_1,t_2) - Ck^2r^{2s}r^n\int_{\R^n\setminus B_r}|y|^{-n-2s}dy\\
&\ge C_0kr^n\tail(u_+;x_0,r,t_1,t_2) - Ck^2r^n. \notag
\end{align}
Similarly, using that $\psi\equiv 0$ in $\R^n\setminus B_{3r/4}$, we have for $x\in B_{3r/4}$ and $y\in \R^n\setminus B_r$ that 
$|x-y|\ge |y| -|x|\ge \frac{|y|}{4}$. This leads to the bound 
\begin{align}\label{32}
I_{32} & \le Ck^2r^n + Ckr^n\tail(u_-;x_0,r,t_1,t_2)\\
&\le Ck^2r^n + Ckr^n\left(\frac{r}{R}\right)^{2s}\tail(u_-;x_0,R,t_1,t_2) ,\notag 
\end{align}
where we also used the assumption on nonnegativity. 
From \eqref{123} and  \eqref{3} we get $$I_{31}\le I_{32}-I_2-I_1.$$  
In combination with \eqref{1}, \eqref{2}, \eqref{31} and \eqref{32}, this leads to  
\begin{align*}
C_0kr^n\tail(u_+;x_0,r,t_1,t_2) \le Ck^2r^n + Ckr^n\tail(u_-;x_0,R,t_1,t_2). 
\end{align*}
We complete the proof by dividing through with $C_0kr^n$. 

\end{proof}

\section{Weak Harnack inequality}\label{sec_weakHarnack}

Our proof of the weak Harnack inequality is based on the approach taken by Moser in \cite{Moser2}. In the case of globally nonnegative supersolutions, it was implemented in the nonlocal setting in \cite{FelKass}. 

We begin with an initial estimate of the local infimum of a supersolution. 
\begin{lemma}\label{lemma_uminusMoser}
Suppose that $u$ is a supersolution to \eqref{maineq} and assume that 
$u\ge 0$ in $B_R(x_0)\times(t_0-r^{2s},t_0)$, where $r<R/2$ and $r^{2s}<t_0 < T$. Let
\[
\tilde u = u + d,\quad\text{where }d \ge \left(\frac{r}{R}\right)^{2s}\tail_\infty(u_-;x_0,R,t_0-r^{2s},t_0). 
\]
Then for any $p>0$ and $\theta\in(0,1)$, there exists a constant $C=C(n,s,\Lambda,p)\ge 1$ such that  
\begin{align}\label{tildeutheta}
\sup_{U^-(x_0,t_0,\theta r)}\tilde u^{-1}\le \frac{C}{(1-\theta)^{\frac{n+2s}{p}}}\left(\fint_{U^-(x_0,t_0,r)}\tilde u^{-p }\right)^{\frac{1}{p}} .
\end{align}
\end{lemma}
\begin{proof}
We set
	\begin{equation*}
		r_0 = r, \quad r_j = \frac{r}{2}(1+2^{-j}), \quad \delta_j = 2^{-j} r,\quad j=1,2,\ldots
	\end{equation*}
	and
	\begin{equation*}
		U_j = B_j \times \Gamma_j = B_{r_j}(x_0) \times (t_0-r_j^{2s}, t_0). 
	\end{equation*}
	We choose nonnegative test functions $\psi_{j} \in C^\infty(B_j)$ and $\zeta_j \in C^{\infty}(\Gamma_j)$ satisfying 
	\begin{equation}\label{deltaspace_inf}
	\psi_j\equiv 1\text{ in }B_{j+1},\;\text{dist}( \supp \psi_j,\R^n\setminus B_j) \ge \frac{\delta_j}{2},
	\end{equation}	
 such that for $\phi_j = \psi_j \zeta_j$ we have
	\begin{equation*}
		0 \leq \phi_j \leq 1, \quad \phi_j=1 \text{ in } U_{j+1},\quad \phi_j(x,t_0-r_j^{2s}) = 0,
	\end{equation*}
	and
	\begin{equation} \label{eqphi_inf}
		|\nabla\phi_j| \leq \frac{C}{r}2^j = C\delta_j^{-1}, \quad \left| \frac{\partial \phi_j}{\partial t}\right| \leq \frac{C}{2s r^{2s}}2^{2sj} = C\delta_j^{-2s}.
	\end{equation}
Let $v = \tilde u^{-\frac{p}{2}}$. 
By the Sobolev embedding theorem (Theorem \ref{parabsobolev}), with $\kappa = \frac{n+2s}{n}$, there holds,  
\begin{align}\label{sobolevUj+1}
&\int_{\Gamma_{j+1}}\fint_{B_{j+1}}|v|^{2\kappa}dxdt\\
&\quad\le Cr_j^{2s-n}\int_{\Gamma_{j+1}}\int_{B_{j+1}}\int_{B_{j+1}}\frac{|v(x,t)-v(y,t)|^2}{|x-y|^{n+2s}}dxdydt\notag\\
&\quad\quad\times\left(\sup_{\Gamma_{j+1}}\fint_{B_{j+1}}|v|^{2}\right)^{\frac{2s}{n}} \notag\\
&  \quad\quad + C\int_{\Gamma_{j+1}}\fint_{B_{j+1}}v^2dxdt\left(\sup_{\Gamma_{j+1}}\fint_{B_{j+1}}|v|^{2}\right)^{\frac{2s}{n}}\notag\\
&\quad\quad = Cr_{j+1}^{2s-n}I_1\times \left(\frac{I_2}{|B_{j+1}|}\right)^{\frac{2s}{n}} + C\int_{\Gamma_{j+1}}\fint_{B_{j+1}}v^2dxdt\times \left(\frac{I_2}{|B_{j+1}|}\right)^{\frac{2s}{n}}, \notag 
\end{align}
where 
\[
I_1 = \int_{\Gamma_{j+1}}\int_{B_{j+1}}\int_{B_{j+1}}\frac{|v(x,t)-v(y,t)|^2}{|x-y|^{n+2s}}dxdydt
\]
and 
\[
I_2 = \sup_{\Gamma_{j+1}}\int_{B_{j+1}}|v|^{2}.
\]
To estimate $I_1$ and $I_2$ we use the Cacciopollo inequality in Lemma \ref{lemma_negativecacc}, with $r=r_j$, $\tau_2=t_2$, $\tau_1 = t_1-r_{j+1}^{2s}$ and $\lag = r_j^{2s}-r_{j+1}^{2s}$. This leads to  
\begin{align}\label{I1I2}
&I_1 + I_2\\
& \le \int_{\Gamma_j}\int_{B_j}\int_{B_j}\psi_j(x)\psi_j(y)\left[\left(\frac{\tilde u(x,t)}{\psi_j(x)}\right)^{-\frac{p}{2}} - \left(\frac{\tilde u(y,t)}{\psi_j(y)}\right)^{-\frac{p}{2}}\right]^2\eta_j(t)d\mu dt\notag\\
& + \sup_{\Gamma_{j+1}}\int_{B_j}\psi_j^{p+2}(x)\tilde u^{-p}(x,t)dx\notag\\
& \le C\int_{\Gamma_j}\int_{B_j}\int_{B_j}(\psi_j(x)-\psi_j(y))^2\left[\left(\frac{\tilde u(x,t)}{\psi_j(x)}\right)^{-p} + \left(\frac{\tilde u(y,t)}{\psi_j(y)}\right)^{-p}\right]\eta_j(t)d\mu dt\notag\\
& + C\sup_{x\in\supp\psi_j}\int_{\R^n\setminus B_j}\frac{dy}{|x-y|^{n+2s}} \int_{\Gamma_{j}}\int_{B_j} \tilde u^{-p}(x,t)\psi^{p+2}(x,t)\eta(t)dxdt \notag\\
& + \frac{C}{d}\sup_{\stackrel{\Gamma_j}{x\in\supp\psi_j}}\int_{\R^n\setminus B_R}\frac{u_-(y,t)dy}{|x-y|^{n+2s}}\int_{\Gamma_j}\int_{B_j} \tilde u^{-p}(x,t)\psi_j^{p+2}(x,t)\eta_j(t)dxdt\notag\\
& + C\int_{\Gamma_j}\int_{B_j}\psi_j^{p+2}(x)\tilde u^{-p}(x,t)\partial_t\eta_j(t)dxdt = J_1 + J_2 + J_3 + J_4, \notag 
\end{align}
where $C\le C_0(n,s,\Lambda)(1+p^2)$. Due to our assumption on $\psi_j$, 
\begin{align}\label{J1}
J_1 & \le 2C\Lambda\frac{2^{2j}}{r^2} \sup_{x\in B_j}\int_{B_j}\frac{|x-y|^2dy}{|x-y|^{n+2s}}\int_{\Gamma_j}\int_{B_j}v^2(x,t)dxdt\\
&\le C\frac{2^{2j}}{r_j^{2s}}\int_{\Gamma_j}\int_{B_j}v^2(x,t)dxdt.\notag
\end{align}
Without loss of generality, it may be assumed that $x_0=0$. 
Recalling \eqref{deltaspace_inf}, we have, for $x\in\supp\psi_j$ and $y\in\R^n\setminus B_j$, 
\[
\frac{1}{|x-y|} = \frac{1}{|y|}\frac{|y|}{|x-y|}\le \frac{1}{|y|}\frac{|x|+|x-y|}{|x-y|} \le \frac{1}{|y|}\left(1 + \frac{r}{\delta_j}\right)\le 2^{j+1}\frac{1}{|y|}. 
\]
Thus 
\begin{align}\label{J2}
J_2 &\le C2^{(n+2s)j}\int_{\R^n\setminus B_j}\frac{dy}{|y|^{n+2s}} \int_{\Gamma_{j}}\int_{B_j} \tilde u^{-p}(x,t)\psi^{p+2}(x,t)\eta(t)dxdt\\
&\le  C2^{(n+2s)j}r_j^{-2s}\int_{\Gamma_{j}}\int_{B_j} v^{2}(x,t)dxdt.\notag 
\end{align}
If $x\in \supp\psi_j\subset B_r$ and $y\in\R^n\setminus B_R$, then 
\[
\frac{1}{|x-y|}\le \frac{1}{|y|}\left(1+\frac{r}{R-r}\right)\le \frac{2}{|y|}. 
\]
Thus $J_3$ satisfies 
\begin{align*}
J_3 & \le \frac{C}{d}R^{-2s}\tail_\infty(u_-;x_0,R,t_0-r^{2s},t_0)\int_{\Gamma_{j}}\int_{B_j} v^{2}(x,t)dxdt.
\end{align*}
Due to our choice of $d$,
\begin{equation}\label{J3}
J_3 \le Cr^{-2s}\int_{\Gamma_{j}}\int_{B_j} v^{2}(x,t)dxdt.
\end{equation}
We finally estimate $J_4$ using the assumption \eqref{eqphi_inf}: 
\begin{align}\label{J4}
J_4 & \le C\frac{2^{2sj}}{r_j^{2s}}\int_{\Gamma_{j}}\int_{B_j} v^{2}(x,t)dxdt.
\end{align}
Using the estimates \eqref{J1}, \eqref{J2}, \eqref{J3} and \eqref{J4} for $J_1$, $J_2$, $J_3$ and $J_4$ in \eqref{I1I2}, we find,  
\begin{align}\label{I1I2final}
I_1 + I_2& \le C2^{(n+2s)j}r_j^{-2s}\int_{\Gamma_{j}}\int_{B_j} v^{2}(x,t)dxdt.
\end{align}
Recalling \eqref{sobolevUj+1}, \eqref{I1I2final} gives
\begin{align}
&\fint_{\Gamma_{j+1}}\fint_{B_{j+1}}|v|^{2\kappa}dxdt \\
&\quad\le r_{j}^{-2s}C
2^{(n+2s)j}\int_{\Gamma_{j}}\fint_{B_j} v^{2}dxdt\left(2^{(n+2s)j}r_j^{-2s}\int_{\Gamma_{j}}\fint_{B_j} v^{2}dxdt\right)^{\frac{2s}{n}}\notag\\
&\quad\le C\left(2^{(n+2s)j}\fint_{\Gamma_{j}}\fint_{B_j} v^{2}dxdt\right)^{\frac{n+2s}{n}}.\notag
\end{align} 
Since $\kappa = \frac{n+2s}{n}$, we have shown that, for any $p>0$,  
\begin{equation}\label{iter1}
\left(\fint_{U_{j+1}}\tilde u^{-p\kappa }\right)^{\frac{1}{p\kappa}}\le 
C^{\frac{1}{p\kappa}}\left(2^{(n+2s)j}\fint_{U_{j}}\tilde u^{-p }\right)^{\frac{1}{p}}. 
\end{equation}
Here $C=C_p$ is increasing in $p$ at a polynomial rate. Let 
\begin{equation*}
A_j = \left(\fint_{U_{j}}\tilde u^{-p_j }\right)^{\frac{1}{p_j}}, \quad p_j = \kappa^jp, \quad \alpha_j = C_{p_j}^{\frac{1}{p_{j+1}}}2^{\frac{(n+2s)j}{p_j}}.
\end{equation*}
Then by \eqref{iter1}, $A_{j+1}\le \alpha_jA_j$ and 
\begin{align*}
A_N \le A_0\prod_{j=0}^{N-1}\alpha_j. 
\end{align*}
It is easy to check that $\prod_{j=0}^{N-1}\alpha_j$ is bounded independently of $N$ by analyzing its logarithm. Hence we obtain 
\begin{align}\label{Urhalf}
\sup_{U(r/2)}\tilde u^{-1}\le \limsup_{N\to\infty}A_N\le C\left(\fint_{U_{(r)}}\tilde u^{-p }\right)^{\frac{1}{p}} 
\le C\left(\fint_{U_{(r)}} u^{-p }\right)^{\frac{1}{p}}. 
\end{align}
If $\theta\in(0,1)$, then 
\begin{align}\label{Uthetar}
\sup_{U(\theta r)}\tilde u^{-1}\le \frac{C}{(1-\theta)^{\frac{n+2s}{p}}}\left(\fint_{U_{(r)}}\tilde u^{-p }\right)^{\frac{1}{p}} 
\le \frac{C}{(1-\theta)^{\frac{n+2s}{p}}}\left(\fint_{U_{(r)}} u^{-p }\right)^{\frac{1}{p}}. 
\end{align}
This is clear if $\theta\le 1/2$. If $\theta>1/2$, choose $(z,\tau)\in U(\theta r)$ such that $U(z,\tau,(1-\theta)r)\subset U(r)$. 
Using \eqref{Urhalf} with $(1-\theta)r$ in place of $r$, we get 
\begin{align*}
\sup_{U(z,\tau,(1-\theta)r/2)}\tilde u^{-1}&\le \frac{C}{(1-\theta)^{\frac{n+2s}{p}}}\left(\frac{1}{|U(r)|}\int_{U_{z,\tau,(1-\theta)r)}}\tilde u^{-p }\right)^{\frac{1}{p}}\\
& \le \frac{C}{(1-\theta)^{\frac{n+2s}{p}}}\left(\fint_{U(r)}\tilde u^{-p }\right)^{\frac{1}{p}}. 
\end{align*}
By covering $U(\theta r)$ with a finite collection of sets $\{U(z_k,\tau_k,(1-\theta)r/2)\}_k$ of the above type, we obtain \eqref{Uthetar}. 

\end{proof}


The next result is a reverse H\"older inequality for supersolutions. 

\begin{lemma}\label{lemma_usmall}
Suppose that $u$ is a supersolution to \eqref{maineq} such that 
\[
u\ge 0\text{ in }B_R(x_0)\times(t_0,t_0+r^{2s}),
\]
where $r<R/2$ and $t_0\in(0,T-r^{2s})$. 
Let
\[
\tilde u = u + d,\quad\text{where }d \ge \left(\frac{r}{R}\right)^{2s}\tail_\infty(u_-;x_0,R,t_0,t_0+r^{2s}). 
\]
Then for any $\hat p\in(0,1)$ and $\theta\in[1/2,1)$, there exist constants $C = C(n,s,\hat p)\ge 1$ and $m=m(s,n)>0$ such that 
\begin{equation}
\fint_{U^+(x_0,t_0,\theta r)}\tilde u dxdt \le
 \left(\frac{C}{(1-\theta)^{m}}\right)^{(1/\hat p-1)}
\left(\fint_{U^+(x_0,t_0,r)}\tilde u^{\hat p}dxdt\right)^{\frac{1}{\hat p}}. 
\end{equation}
\end{lemma}

\begin{proof}
The proof is very similar to that of Lemma \ref{lemma_uminusMoser} and we only provide enough details to follow the main ideas. See also Theorem 3.7. in \cite{FelKass}. Let  
\[
r_0 = r,\quad r_j = r_j = r-(1-\theta)2^{-j}, \quad \delta_j = (1-\theta)2^{-j}r,\quad j=1,2,\ldots
\]
and 
\[
U_{j} = B_j\times\Gamma_j = B_{r_{j}}(x_0)\times(t_0,t_0+r_j^{2s}). 
\]
We choose nonnegative test functions $\psi_{j} \in C^\infty(B_j)$ and $\zeta_j \in C^{\infty}(\Gamma_j)$ satisfying 
	\begin{equation}\label{deltaspace}
	\psi_j\equiv 1\text{ in }B_{j+1},\;\text{dist}( \supp \psi_j,\R^n\setminus B_j) \ge \frac{\delta_j}{2},
	\end{equation}	
 such that for $\phi_j = \psi_j \zeta_j$ we have
	\begin{equation*}
		0 \leq \phi_j \leq 1, \quad \phi_j=1 \text{ in } U_{j+1},\quad \phi_j(x,t_0+r_j^{2s}) = 0,
	\end{equation*}
	and
	\begin{equation} \label{eqphi}
		|\nabla\phi_j| \leq \frac{C}{r}2^j = C\delta_j^{-1}, \quad \left| \frac{\partial \phi_j}{\partial t}\right| \leq \frac{C}{2s r^{2s}}2^{2sj} = C\delta_j^{-2s}.
	\end{equation}
For $p\in(0,1)$, let $v = \tilde u^{\frac{p}{2}}$. At this point the proof proceeds exactly as the proof of Lemma \ref{lemma_uminusMoser}: We use the parabolic Sobolev inequality, this time using Lemma \ref{lemma_positivecacc} with $r=r_j$, $\tau_1 = t_0$, $\tau_2=t_0+r_{j+1}^{2s}$ and $\lag = r_{j}^{2s} - r_{j+1}^{2s}$, to estimate $I_1 + I_2$. Completely analogously to the proof of Lemma \ref{lemma_positivecacc}, we obtain 
\[
I_1 + I_2 \le C2^{(n+2s)j}r_j^{-2s}\int_{\Gamma_j}\int_{B_j}v^2dxdt, 
\]
where $C = C(p)$ and 
\begin{equation}\label{rhrh}
\left(\fint_{U_{j+1}}|\tilde u|^{p\kappa}\right)^{\frac{1}{p\kappa}} \le 
C^{\frac{1}{p\kappa}}\left(\frac{2^{(n+2s)j}}{(1-\theta)^{n+2s}}\fint_{U_{j}}|\tilde u|^{p}\right)^{\frac{1}{p}}. 
\end{equation}
For $j$ such that $\kappa^jp < 1$, let 
\begin{equation}\label{Ajpj}
A_j = \left(\fint_{U_{j}}\tilde u^{p_j }\right)^{\frac{1}{p_j}}, \quad p_j = \kappa^jp, \quad \alpha_j = C_{p_j}^{\frac{1}{p_{j+1}}}\frac{2^{\frac{(n+2s)j}{p_j}}}{(1-\theta)^{\frac{n+2s}{p_j}}}.
\end{equation}
Choose $N$ such that $\kappa^{N-1}p<1$. Then if $0\le j\le N-1$, we have $p\le p_j\le \kappa^{N-1}p$. Thus by Lemma \ref{lemma_positivecacc}, the constant $C_{p_j}$ depends on $p$ and $\kappa^{N-1}p$ only.   
From the construction in \eqref{Ajpj}, we obtain from \eqref{rhrh} that  $A_{j+1}\le \alpha_jA_j$ and, if $\kappa^{N-1}p<1$,  
\begin{align}
A_N \le A_0\prod_{j=0}^{N-1}\alpha_j\le C_1^{\frac{1}{p}}\frac{1}{(1-\theta)^{m_0(\kappa^N-1)}}A_0,  
\end{align}
for some $C_1 = C_1(n,s,\Lambda,p,\kappa^{N-1}p)$ and $m_0= m(n,s)$. 
Choosing $N$ so that $\kappa^{-N}\le \hat p\le \kappa^{-N+1}$ and setting $p=\kappa^{-N}$, we get 
\begin{align}\label{rhpos12}
\fint_{U^+(\theta r)}\tilde u dxdt& \le A_N\le  \frac{C_1^{\kappa^N}}{(1-\theta)^{m(1/\hat p-1)}}
\left(\fint_{U^+(r)}\tilde u^{\kappa^{-N}}dxdt\right)^{\frac{1}{\kappa^{-N}}}\\
& \le \frac{C_1^{\kappa^N}}{(1-\theta)^{m(1/\hat p-1)}}
\left(\fint_{U^+(r)}\tilde u^{\hat p}dxdt\right)^{\frac{1}{\hat p}} = \notag\\
&\le \frac{C^{1/\hat p-1}}{(1-\theta)^{m(1/\hat p-1)}}
\left(\fint_{U^+(r)}\tilde u^{\hat p}dxdt\right)^{\frac{1}{\hat p}},  
\end{align}
where $C = C(n,s,\Lambda,\hat p,\kappa^{-1}) = C(n,s,\Lambda,\hat p)$ and $m=m(n,s)$. 


\end{proof}

\subsection{Logarithmic estimates}

Here we prove logarithmic estimates for supersolutions. Together with Lemma 
\ref{lemma_uminusMoser} and Lemma \ref{lemma_usmall}, these estimates enables us to use an abstract lemma (Lemma \ref{lemma_SC} at the end of this section) proved by Moser in \cite{Moser2}. The proofs follow closely those of Lemma 4.1. and Proposition 4.2. in \cite{FelKass}, though additional care is required here to handle the negative parts of the supersolutions. 

\begin{lemma}\label{lemma_logest}
For $0<r<R/2$, let $B_r$ and $B_R$ be concentric balls in $\R^n$. Assume that $u\ge 0$ in $B_R\times(t_1,t_2)$. 
Let $\psi\in C_c^\infty(B_r)$ be a nonnegative function such that 
$\psi\le 1$ and $|\nabla \psi|\le Cr^{-1}$.  
Then 
\begin{align*}
&\int_{t_1}^{t_2}\E(\tilde u,-\psi^2\tilde u^{-1},t)dt \\
& \ge \int_{t_1}^{t_2}\int_{B_r}\int_{B_r}\psi(x)\psi(y)\left|\log\frac{\tilde u(x,t)}{\psi(x)} - \log\frac{\tilde u(y,t)}{\psi(y)}\right|^2d\mu(x,y,t)dt\\
& - Cr^{n-2s}(t_2-t_1) - \frac{2\Lambda}{d}\int_{t_1}^{t_2}\int_{\R^n\setminus B_R}\int_{B_r}\frac{u_-(y,t)}{|x-y|^{n+2s}}dxdydt.
\end{align*}

$\tilde u = u+d$. 

\end{lemma}
\begin{proof}
Due to the assumption on nonnegativity, we have 
\begin{align}\label{logI1I2}
&\int_{t_1}^{t_2}\E(\tilde u,-\psi^2\tilde u^{-1},t)dt\\
&\ge \int_{t_1}^{t_2}\int_{B_r}\int_{B_r}F(x,y,t)d\mu(x,y,t)dt\notag\\
& + 2\int_{t_1}^{t_2}\int_{\R^n\setminus B_r}\int_{B_r}(\tilde u(x,t)-\tilde u(y,t))\left(\frac{\psi^2(y)}{\tilde u(y,t)} - \frac{\psi^2(x)}{\tilde u(x,t)}\right)d\mu(x,y,t)dt\notag\\
&= I_1 + I_2, \notag
\end{align}
where 
\[
F(x,y,t) = \psi(x)\psi(y)\left(\frac{\psi(x)\tilde u(y,t)}{\psi(y)\tilde u(x,t)} + \frac{\psi(y)\tilde u(x,t)}{\psi(x)\tilde u(y,t)} - \frac{\psi(y)}{\psi(x)} - \frac{\psi(x)}{\psi(y)}\right). 
\]
Since $u\ge 0$ in $B_R$ and $\psi=0$ in $\R^n\setminus B_r$ we have 
\begin{align}\label{logI2}
I_2 & \ge - 2\int_{t_1}^{t_2}\int_{\R^n\setminus B_r}\int_{B_r}\psi^2(x)d\mu(x,y,t)dt\\
&-2\int_{t_1}^{t_2}\int_{\R^n\setminus B_r}\int_{B_r}\psi^2(x)\frac{\tilde u_-(y,t)}{v(x,t)}d\mu(x,y,t)dt\notag\\
&\ge -Cr^{n-2s}(t_2-t_1) - \frac{2\Lambda}{d}\int_{t_1}^{t_2}\int_{\R^n\setminus B_R}\int_{B_r}\frac{u_-(y,t)}{|x-y|^{n+2s}}dxdydt.\notag 
\end{align}
Arguing as in \cite{FelKass}, Lemma 4.1., it can be shown that 
\begin{align}\label{logI1}
I_1 & \ge \int_{t_1}^{t_2}\int_{B_r}\int_{B_r}\psi(x)\psi(y)\left|\log\frac{\tilde u(x,t)}{\psi(x)} - \log\frac{\tilde u(y,t)}{\psi(y)}\right|^2d\mu(x,y,t)dt\\
&-\int_{t_1}^{t_2}\int_{B_r}\int_{B_r}|\psi(x)-\psi(y)|^2d\mu(x,y,t)dt\notag\\
& \ge \int_{t_1}^{t_2}\int_{B_r}\int_{B_r}\psi(x)\psi(y)\left|\log\frac{\tilde u(x,t)}{\psi(x)} - \log\frac{\tilde u(y,t)}{\psi(y)}\right|^2d\mu(x,y,t)dt\notag\\
&\quad -Cr^{n-2s}(t_2-t_1). \notag
\end{align}
Here we used the fact that 
\begin{align*}
&\int_{t_1}^{t_2}\int_{B_r}\int_{B_r}|\psi(x)-\psi(y)|^2d\mu(x,y,t)dt\\
&\le \frac{C\Lambda}{r^2}\int_{t_1}^{t_2}\int_{B_r}\int_{B_r}|x-y|^{2-n-2s}dxdtdt\le Cr^{n-2s}(t_2-t_1). 
\end{align*}
Using \eqref{logI1} and \eqref{logI2} in \eqref{logI1I2}, we complete the proof. 
\end{proof}

With the aid of Lemma \ref{lemma_logest}, we derive estimates for the levelsets of the logarithm of a supersolution.  

\begin{lemma}\label{lemma_levelsets}
Let $x_0\in\R^n$, $r>0$ and $t_0\in (r^{2s},T-r^{2s})$. Suppose that $u$ is a supersolution to \eqref{maineq} such that 
\[
u\ge 0\text{ in }B_R(x_0)\times(t_0-r^{2s},t_0+r^{2s}), \quad 0<r<R/2.
\]
Let 
\[
\tilde u = u + d,\quad \text{where } d = \left(\frac{r}{R}\right)^{2s}\tail_\infty(u_-;x_0,R,t_0-r^{2s},t_0+r^{2s}). 
\]
Then there exists a constant $C = C(n,s,\Lambda)$ such that 
\begin{align}
&|U^+(x_0,t_0,r)\cap\{\log\tilde u<-\gamma-a\}| \le \frac{C|U^+(x_0,t_0,r)|}{\gamma}\le 
\frac{Cr^{n+2s}}{\gamma},\tag{i}\\
&|U^-(x_0,t_0,r)\cap\{\log\tilde u > \gamma-a\}| \le \frac{C|U^-(x_0,t_0,r)|}{\gamma}\le 
\frac{Cr^{n+2s}}{\gamma},\tag{ii}
\end{align}
where $a = a(\tilde u(\cdot,t_0))$. 
\end{lemma}

\begin{proof}
We first prove (i). It may be assumed that $x_0=0$. 
Let $\psi(x)\in C_c^\infty(B_{3r/2})$ be a non negative function such that $\psi\equiv 1$ in $B_r$. We additionally assume that there exists a monotone nonincreasing function $\Psi$ such that $\psi(x) = \Psi(|x-x_0|)$. Set $\tilde u = u+d$ and $\phi(x,t) = \frac{\psi^2(x)}{\tilde u(x,t)}$. Let 
\[
t_1 = t_0-r^{2s},\quad t_2 = t_0 + r^{2s}. 
\]
We use $\phi$ as test function and obtain, with $v(x,t) = -\log\frac{\tilde u(x,t)}{\psi(x)}$, 
\begin{align*}
&\int_{t_0}^{\tau}\int_{B_{3r/2}}\psi^2(x)\partial_tv(x,t)dxdt + \int_{t_0}^{\tau}\E(\tilde u,-\psi^2\tilde u^{-1},t)dt \le 0,  
\end{align*}
for any $\tau\in (t_0,t_2)$. 
From Lemma \ref{lemma_logest} we get 
\begin{align}\label{lvlset1}
&\int_{t_0}^{\tau}\int_{B_{3r/2}}\psi^2(x)\partial_tv(x,t)dx\\
&+ \int_{t_0}^{\tau}\int_{B_{3r/2}}\int_{B_{3r/2}}\psi(x)\psi(y)\left|v(x,t) - v(y,t)\right|^2d\mu(x,y,t)\notag\\
& \le Cr^{n-2s}(\tau-t_0) + \frac{2\Lambda}{d}\int_{t_0}^{\tau}\int_{\R^n\setminus B_R}\int_{B_{3r/2}}\frac{u_-(y,t)}{|x-y|^{n+2s}}dxdydt.\notag
\end{align}
Let 
\[
V(t) = \frac{\int_{B_{3r/2}}v(x,t)\psi^2(x)dx}{\int_{B_{3r/2}}\psi^2(x)dx}.
\]
Then an application of the weighted Poincar\'e inequality in Lemma \ref{DK},  to the second term on the left hand side in \eqref{lvlset1}, yields
\begin{align}\label{lvlset4}
& \int_{t_0}^{\tau}\int_{B_{3r/2}}\psi^2(x)\partial_tv(x,t)dxdt\\
& + cr^{-2s}\int_{t_0}^{\tau}\int_{B_{3r/2}}\psi^2(x)\left|v(x,t) - V(t)\right|^2dxdt \notag\\
&\le Cr^{n-2s}(\tau-t_0) + \frac{2\Lambda}{d}\int_{t_0}^{\tau}\int_{\R^n\setminus B_R}\int_{B_{3r/2}}\frac{u_-(y,t)}{|x-y|^{n+2s}}dxdydt \notag\\
& \le Cr^{n-2s}(\tau-t_0) + \frac{C\Lambda}{dR^{2s}}\int_{t_0}^{\tau}\int_{B_{3r/2}}R^{2s}\int_{\R^n\setminus B_R}\frac{u_-(y,t)}{|y|^{n+2s}}dydxdt \notag\\
& \le Cr^{n-2s}(\tau-t_0)\left(1 + \frac{1}{d}\left(\frac{r}{R}\right)^{2s}\tail_\infty(u_-;0,R,t_0,t_2)\right) \notag\\
&\le  Cr^{n-2s}(\tau-t_0). \notag
\end{align}
Since $\psi \equiv 1$ in $B_r$ and $\int_{B_{3r/2}}\psi^2(x)dx\approx r^n$, we obtain after dividing through with $\int_{B_{3r/2}}\psi^2(x)dx$ in \eqref{lvlset4}, 
\begin{align}\label{lvlset5}
& V(\tau)-V(t_0) + cr^{-2s}\int_{t_0}^{\tau}\fint_{B_{r}}\left|v(x,t) - V(t)\right|^2dx\\
& \le Cr^{-2s}(\tau-t_0) \notag
\end{align}
By Remark \ref{rmk1}, we may assume $V$ to be continuous on $[t_1,t_2]$
Choose $\delta>0$ such that if $t_0\le \theta_1\le \theta_2 \le t_2$, and $\theta_2-\theta_1\le \delta$, then 
\[
|V(\theta_2)-V(\theta_1)|^2\le 1. 
\]
Let $\tau = t_0+\delta$. Then  
\begin{align}\label{lvlset6}
|v(x,t)-V(\tau)|^2\le 2|v(x,t)-V(t)| + \e, 
\end{align}
for all $t\in (t_0,\tau)$. Hence, using \eqref{lvlset6} in \eqref{lvlset5},
\begin{align*}
& V(\tau)-V(t_0) + cr^{-2s}\int_{t_0}^{\tau}\fint_{B_{r}}\left|v(x,t) - V(\tau)\right|^2dx\\
& \le C_1r^{-2s}(\tau-t_0).
\end{align*}
We now set 
\begin{equation*}
w(x,t) = v(x,t)-C_1t,\quad W(t) = V(t) - C_1t,\quad\text{and }a=V(t_0).  
\end{equation*}
Then 
\begin{align}
& W(\tau)-W(t_0) + cr^{-2s}\int_{t_0}^{\tau}\fint_{B_{r}}\left|w(x,t) - W(\tau) + C_1(\tau-t)\right|^2dx\le 0, \label{eqtaut}
\end{align}
so that $W$ is nonincreasing on $[t_0,\tau]$. Let 
\[
L_\gamma(t) = \{x\in B_r: w(x,t)>\gamma+a-C_1t_0\}. 
\]
Then it follows from \eqref{eqtaut} that
\begin{align*}
& W(\tau)-W(t_0) + cr^{-2s}\int_{t_0}^{\tau}r^{-n}\int_{L_\gamma(t)}\left|w(x,t) - W(\tau) + C_1(\tau-t)\right|^2dx\le 0.  
\end{align*}
Additionally, when $x\in L_\gamma(t)$, 
\begin{equation}\label{ges}
w(x,t)-W(\tau)>\gamma+a-C_1t_0-W(\tau) \ge \gamma+a-C_1t_0-W(t_0) = \gamma>0.  
\end{equation}
Thus we find 
\begin{align*}
& W(\tau)-W(t_0) + cr^{-2s}\int_{t_0}^{\tau}r^{-n}\int_{L_\gamma(t)}\left|w(x,t) - W(\tau)\right|^2dx\le 0,   
\end{align*}
which yields 
\begin{align}\label{W12}
&\frac{W(\tau)-W(t_0)}{(\gamma+a-C_1t_0-W(\tau))^2} + cr^{-2s}\frac{\int_{t_0}^{\tau}|L_\gamma(t)|dt}{|B_r|}\le 0. 
\end{align}
Using that $W(\tau)<W(t_0)$ we deduce from \eqref{W12} that 
\begin{align*}
cr^{-2s}\frac{\int_{t_0}^{\tau}|L_\gamma(t)|dt}{|B_r|}&\le 
\frac{W(t_0) - W(\tau)}{(\gamma+a-C_1t_0-W(\tau))(\gamma+a-C_1t_0-W(t_0))} \\
& = \frac{1}{\gamma+a-C_1t_0-W(t_0)} - \frac{1}{\gamma+a-C_1t_0-W(\tau)}. 
\end{align*}
We decompose the interval $(t_0,t_2)$ as $\cup_i(\tau_i,\tau_{i+1})$, where 
\[\tau_1=t_0,\; \tau_N = t_2,\; N=[\delta^{-1}(t_2-t_0)]\; \text{and } \tau_{i+1}=\tau_{i}+\delta \text{ for }
i=1,\ldots, N-2.  
\]
For each interval $(\tau_i,\tau_{i+1})$ we get
\begin{align*}
&cr^{-2s}\frac{\int_{\tau_i}^{\tau_{i+1}}|L_\gamma(t)|dt}{|B_r|} 
\le\frac{1}{\gamma+a-C_1t_0-W(\tau_i)} - \frac{1}{\gamma+a-C_1t_0-W(\tau_{i+1})}. 
\end{align*}
Thus the sum $\sum_{i=1}^{M-1}\int_{\tau_i}^{\tau_{i+1}}|L_\gamma(t)|dt$ telescopes and we obtain, using \eqref{ges},  
\begin{align}\label{intLs}
&cr^{-2s}\frac{\int_{t_0}^{t_2}|L_\gamma(t)|dt}{|B_r|}
\le \frac{1}{\gamma+a-C_1t_0-W(t_0)} - \frac{1}{\gamma+a-C_1t_0-W(t_2)}\le \frac{1}{\gamma}. 
\end{align}
From \eqref{intLs} we deduce 
\begin{align*}
&|B_r\times(t_0,t_2)\cap\{w>\gamma+a-C_1t_0\}|\le \frac{Cr^{n+2s}}{\gamma}. 
\end{align*}
Going back to $\tilde u$, we find that 
\begin{align*}
&|\{B_r\times(t_0,t_2)\cap\{\log\tilde u < -\gamma-a\}\}| \\
&\le |\{B_r\times(t_0,t_2)\cap\{\log\tilde u +C_1(t-t_0)<-\gamma/2-a\}\}|\\
&\quad+|\{B_r\times(t_0,t_2)\cap \{C_1(t-t_0)>\gamma/2\}|\\
&=|\{B_r\times(t_0,t_2)\cap\{w > \gamma/2+a-C_1t_0\}\}|\\
&\quad+|\{B_r\times(t_0,t_2)\cap \{C_1(t-t_0)>\gamma/2\}|\\
&\quad\le \frac{Cr^{n+2s}}{\gamma} + r^{n+2s}\left(1-\frac{\gamma}{2C_1r^{2s}}\right)\le \frac{Cr^{n+2s}}{\gamma}. 
\end{align*}
This completes the proof of (i). To prove (ii), we proceed analogously, but initially integrate from $\tau\in(t_1,t_0)$ up to $t_0$. In this case we define $L_\gamma(t)$ in terms of the inequality $w<a-\gamma-C_1t_0$.

\end{proof}

The lemma below can be found in \cite{SalCos}, Section 2.2.3. It enables us to prove the weak Harnack inequality using the previous results of this section.  

\begin{lemma}\label{lemma_SC}
Let $\{U(\theta r)\}_{1/2\le \theta\le 1}$ be a family of non decreasing domains in $\R^{n+1}$. Let $m,C_0$ be positive constants, let $\sigma\in(0,1)$ and let $p_0\in(0,\infty]$. Suppose that $w$ is a non negative function satisfying 
\begin{equation*}
|U(r)\cap \{\log w>\gamma\}| \le \frac{C_0}{\gamma}|U(r)|
\end{equation*}
and 
\begin{equation*}
\left(\fint_{U(\theta r)}w^{p_0}dxdt\right)^{\frac{1}{p_0}}
\le \left(\frac{C_0}{(1-\theta)^m}\right)^{\frac{1}{p}-\frac{1}{p_0}}
\left(\fint_{U(r)}w^p\right)^{\frac{1}{p}}, 
\end{equation*}
for all $p\in (0,\min\{1,\sigma p_0\})$. 
Then there exists a constant $C = C(\sigma,\theta,m,C_0,p_0)$ such that 
\begin{equation*}
\left(\fint_{U(\theta r)}w^{p_0}dxdt\right)^{\frac{1}{p_0}}\le C. 
\end{equation*}
\end{lemma}

\subsection*{Proof of Theorem \ref{weakHarnack}}

\begin{proof}
Assume $u\ge 0$ in $B_R\times(t_0-r^{2s},t_0+r^{2s})$ and set 
\[
d = \left(\frac{r}{R}\right)^{2s}\tail_\infty(u_-;x_0,t_0-r^{2s},t_0+r^{2s}),\quad \tilde u = u + d. 
\]
Let 
\begin{align*}
&U_1(\theta r) = B_{\theta r}\times(t_0+r^{2s}-(\theta r)^{2s},t_0+r^{2s}) = 
U^-(x_0,t_0+r^{2s},\theta r),\\ 
&U_2(\theta r) = B_{\theta r}\times(t_0-r^{2s},t_0-r^{2s} + (\theta r)^{2s}) = U^+(x_0,t_0-r^{2s},\theta r). 
\end{align*}
We note that $U_1(r) = U^+(x_0,t_0,r)$ and $U_2(r) = U^-(x_0,t_0,r)$. 
Let $a = a(\tilde u(\cdot,t_0))$ be the constant in Lemma \ref{lemma_levelsets} and set $w_1 = e^{-a}\tilde u^{-1}$, $w_2 = e^{a}\tilde u$. Then by Lemma \ref{lemma_levelsets}, 
\begin{equation}\label{loglvl}
|U_i(r)\cap \{\log w_i>\gamma\}|\le \frac{C|U_i(r)|}{\gamma},\quad i = 1,2.
\end{equation}
From Lemma \ref{lemma_uminusMoser} we obtain, for any $p>0$,  
\begin{align}\label{supw1}
\sup_{U_1(\theta r)}w_1 & = e^{-a}\sup_{U_1(\theta r)}\tilde u^{-1}\le \frac{Ce^{-a}}{(1-\theta)^{\frac{n+2s}{p}}}\left(\fint_{U_1{(r)}}\tilde u^{-p }\right)^{\frac{1}{p}}\\
& = \frac{C}{(1-\theta)^{\frac{n+2s}{p}}}\left(\fint_{U_1{(r)}}w_1^p\right)^{\frac{1}{p}}.\notag 
\end{align}
An application of Lemma \ref{lemma_usmall} gives, for any $\hat p\in(0,1)$, 
\begin{align}\label{meanw2}
\fint_{U_2(\theta r)}w_2 dxdt & =e^a\fint_{U_2(\theta r)}\tilde u dxdt
\le
 \frac{Ce^a}{(1-\theta)^{m(1/\hat p-1)}}
\left(\fint_{U_2(r)}\tilde u^{\hat p}dxdt\right)^{\frac{1}{\hat p}}\\
& = \frac{C}{(1-\theta)^{m(1/\hat p-1)}}
\left(\fint_{U_2(r)}w_2^{\hat p}dxdt\right)^{\frac{1}{\hat p}}. \notag
\end{align} 
With \eqref{loglvl} and \eqref{supw1} at hand, we apply Lemma \ref{lemma_SC} to $w = w_1$ with $p_0 = \infty$ and any $\sigma\in (0,1)$, to find 
\begin{equation}\label{w_1}
\sup_{U_1(\theta_1 r)}w_1 \le C_1(\theta_1),\quad \frac12\le \theta_1<1. 
\end{equation}
Setting $w=w_2$, $p_0=1$ and again any $\sigma\in(0,1)$, we get, again from Lemma \ref{lemma_SC},  
\begin{equation}\label{w_2}
\fint_{U_2(\theta_2 r)}w_2dxdt \le C_2(\theta_2), \quad\frac12\le \theta_2<1. 
\end{equation}
Let $r_i=\theta_ir$, $i=1,2$. From \eqref{w_1} and \eqref{w_2} we find 
\begin{align*}
e^{a}\fint_{U_2(r_2)}udxdt& \le \fint_{U_2(r_2)}w_2dxdt \le C_1\\
& \le \frac{C_1C_2}{\sup_{U_1(r_1)}w_1} = C_1C_2e^a\left(\inf_{U_1(r_1)}u + d\right). 
\end{align*}
Thus we arrive at the weak Harnack inequality
\begin{align}\label{whr1r2}
&\fint_{B_{r_2}\times(t_0-r^{2s},t_0-r^{2s}+r_2^{2s})}udxdt
\le C\inf_{B_{r_1}\times(t_0+r^{2s}-r_1^{2s},t_0+r^{2s})}u \\
&\qquad + C\left(\frac{r}{R}\right)^{2s}\tail_\infty(u_-;x_0,R,t_0-r^{2s},t_0+r^{2s}).\notag
\end{align}
Let $r_1=r_2=\rho$ in \eqref{whr1r2} and choose $\rho$ so that $r^{2s} = 2\rho^{2s}$. This leads to the desired inequality  
\begin{align*}
&\fint_{B_\rho\times(t_0-2\rho^{2s},t_0-\rho^{2s})}udxdt\\
&\le C\left(\inf_{B_\rho\times(t_0+\rho^{2s},t_0+2\rho^{2s})}u +  \left(\frac{\rho}{R}\right)^{2s}\tail_\infty(u_-;x_0,R,t_0-2\rho^{2s},t_0+2\rho^{2s})\right).
\end{align*}

\end{proof}

\section{Local boundedness}\label{sec_localbound}

We start with two classical technical lemmas that are needed for the proof. 

\begin{lemma}[see Lemma 4.3 in \cite{HanLin}]\label{lemma_HL}
Let $f(\theta)$ be a non negative bounded function on $[1/2,1]$. Suppose there exist nonnegative constants $C_1,C_2,\alpha,\beta$ where $\beta<1$, such that for any $1/2\le \theta<\sigma\le 1$, there holds
\[
f(\theta) \le C_1(\sigma-\theta)^{-\alpha} + C_2 + \beta f(\sigma). 
\]
Then there exists a constant $C$ depending only on $\alpha$ and $\beta$ such that 
\[
f(\theta) \le C\left(C_1(\sigma-\theta)^{-\alpha}+C_2\right). 
\]
\end{lemma}

\begin{lemma}[See Lemma 4.1 in \cite{DiB}]\label{lemma_iteration}
Let $\{Y_j\}_{j=0}^\infty$ be a sequence of real positive numbers satisfying 
\[
Y_{j+1} \le c_0b^jY_{j}^{1+\beta}, 
\]
for some constants $c_0>0$, $b>1$ and $\beta>0$. Then if $Y_0\le c_0^{-\frac{1}{\beta}}b^{-\frac{1}{\beta^2}}$, 
\[
\lim_{j\to\infty}Y_j=0. 
\]
\end{lemma}

\begin{proposition}\label{prop_localbound}
Suppose that $u$ is a subsolution to \eqref{maineq}. For $x_0\in\R^n$,  $r>0$ and $t_0\in(r^{2s},T)$, set $U^-(r) = U^-(x_0,t_0,r)$.  
Then for any $\theta\in(0,1)$ and any $\delta\in(0,1)$, 
\begin{align*}
\sup_{U^-(\theta r)}u &\le  \frac{C\delta^{-\frac{n+2s}{2s}}}{(1-\theta)^{\frac{n+2s}{2}}}\left(\fint_{U^-(r)}u_+^2dxdt\right)^{\frac12}\\
&+ \delta \tail_\infty(u_+;x_0,r/2,t_0-r^{2s},t_0).
\end{align*}
\end{proposition}

\begin{proof}
We give the proof for $\theta = 1/2$. The general assertion then follows from a covering argument. 
Let 
	\begin{align*}
		&r_0 = r, \quad r_j = \frac{r}{2}(1+2^{-j}),\quad \tilde r_j = \frac{r_j+r_{j+1}}{2} \quad \delta_j = 2^{-j-3} r, \quad j=1,2,\ldots
	\end{align*}
	and let 
	\begin{align*}
		&U_j = B_j \times \Gamma_j = B_{r_j}(x_0) \times (t_0-r_j^{2s}, t_0),\\
        &\tilde U_j = \tilde B_j \times \tilde \Gamma_j = B_{\tilde r_j}(x_0) \times (t_0-\tilde r_j^{2s}, t_0).
	\end{align*}
	We choose nonnegative test functions $\psi_{j} \in C_c^\infty(\tilde B_j)$ and $\zeta_j \in C^{\infty}(\Gamma_j)$ satisfying 
	\begin{equation}\label{subdeltaspace}
	\psi_j\equiv 1\text{ in }B_{j+1},\quad \zeta_j\equiv 1\text{ on }\Gamma_{j+1},\quad \zeta_j\equiv 0\text{ on }\Gamma_{j}\setminus\tilde\Gamma_j
	\end{equation}	
 such that 
	\begin{equation} \label{subeqphi}
		|\nabla\psi_j| \leq \frac{C}{r}2^{j+3} = C\delta_j^{-1}, \quad \left| \frac{\partial \zeta_j}{\partial t}\right| \leq \frac{C}{2s r^{2s}}2^{2s(j+3)} = C\delta_j^{-2s}.
	\end{equation}
For 
\begin{align}\label{kge}
k\in \R\quad\text{and }\tilde k\ge \delta\tail_\infty(u_+;x_0,r/2,t_0-r^{2s},t_0),
\end{align}
we set 
\begin{align*}
&k_j = k + (1-2^{-j})\tilde k,\quad \tilde k_j = \frac{k_{j+1}+k_j}{2},\\
&w_j = (u-k_j)_+,\quad \tilde w_j = (u-\tilde k_j)_+. 
\end{align*} 
We note that since $\tilde k_j>k_j$, we have $w_j\ge \tilde w_j$.  Thus if $\tilde w_j>0$, then $u>\tilde k_j$ and so $w_j=u-k_j>\tilde k_j-k_j$. It follows that
\begin{equation}\label{tilde}
2^{-j-2}\tilde k\tilde w_j = (\tilde k_j-k_j)\tilde w_j\le w_j^2.
\end{equation}

By the Sobolev embedding theorem, with $\kappa = \frac{n+2s}{n}$, there holds,  
\begin{align}\label{subsobolevUj+1}
&\int_{\Gamma_{j+1}}\fint_{B_{j+1}}|\tilde w_j|^{2\kappa}dxdt\\
&\quad\le Cr_j^{2s-n}\int_{\Gamma_{j+1}}\int_{B_{j+1}}\int_{B_{j+1}}\frac{|\tilde w_j(x,t)-\tilde w_j(y,t)|^2}{|x-y|^{n+2s}}dxdydt\notag\\
&\quad\quad\times\left(\sup_{\Gamma_{j+1}}\fint_{B_{j+1}}|\tilde w_j|^{2}\right)^{\frac{2s}{n}} \notag\\
&  \quad\quad + C\int_{\Gamma_{j+1}}\fint_{B_{j+1}}\tilde w_j^2dxdt\left(\sup_{\Gamma_{j+1}}\fint_{B_{j+1}}|\tilde w_j|^{2}\right)^{\frac{2s}{n}}\notag\\
&\quad\quad = Cr_{j+1}^{2s-n}I_1\times \left(\frac{I_2}{|B_{j+1}|}\right)^{\frac{2s}{n}} + C\int_{\Gamma_{j+1}}\fint_{B_{j+1}}\tilde w_j^2dxdt\times \left(\frac{I_2}{|B_{j+1}|}\right)^{\frac{2s}{n}}, \notag 
\end{align}
where 
\[
I_1 = \int_{\Gamma_{j+1}}\int_{B_{j+1}}\int_{B_{j+1}}\frac{|\tilde w_j (x,t)-\tilde w_j(y,t)|^2}{|x-y|^{n+2s}}dxdydt
\]
and 
\[
I_2 = \sup_{\Gamma_{j+1}}\int_{B_{j+1}}|\tilde w_j|^{2}.
\]
To estimate $I_1$ and $I_2$ we use the Cacciopollo inequality in Lemma \ref{lemma_subsolcacc}, with $r=\tilde r_j$, $\tau_2=t_0$, $\tau_1 = t_0-r_{j+1}^{2s}$ and $\ell = \tilde r_j^{2s}-r_{j+1}^{2s}$. This leads to  
\begin{align}\label{subI1I2}
&I_1 + I_2\\
& \le \int_{\tilde\Gamma_j}\int_{\tilde B_j}\int_{\tilde B_j}|\tilde w_j(x,t)\psi_j(x)-\tilde w_j(y,t)\psi_j(y)|^2\eta_j^2(t)d\mu dt\notag\\
& + \frac{1}{2}\sup_{\Gamma_{j+1}}\int_{\tilde B_j}\tilde w_j^2(x,t)\psi_j^2(x)dx \notag\\
& \le C\int_{\Gamma_j}\int_{ B_j}\int_{B_j}\max\{\tilde w_j^2(x,t),\tilde w_j^2(y,t)\}|\psi_j(x)-\psi_j(y)|^2\eta_j^2(t)d\mu dt\notag\\
& + C\sup_{\substack{t\in \Gamma_j\\x\in\supp\psi_j}}\int_{\R^n\setminus B_j}\frac{(\tilde w_j)_+(y,t)dy}{|x-y|^{n+2s}}\int_{\Gamma_j}\int_{B_j}\tilde w_j(x,t)\psi_j^2(x)\eta_j^2(t)dxdt\notag\\
& + \frac{1}{2}\int_{\Gamma_j}\int_{B_j}\tilde w_j^2(x,t)\psi_j^2(x)\partial_t\eta_j^2(t)dxdt = J_1 + J_2 + J_3.\notag 
\end{align}
Due to our assumption on $\psi_j$, 
\begin{align}\label{subJ1}
J_1 & \le 2C\Lambda\frac{2^{2(j+3)}}{r^2} \sup_{x\in \tilde B_j}\int_{B_j}\frac{|x-y|^2dy}{|x-y|^{n+2s}}\int_{\Gamma_j}\int_{B_j}\tilde w_j^2(x,t)dxdt\\
&\le C\frac{2^{2(j+3)}}{r_j^{2s}}\int_{\Gamma_j}\int_{B_j}w_j^2(x,t)dxdt.\notag
\end{align}
To estimate $J_2$, we first observe that, due to \eqref{tilde}, 
\begin{align}\label{subJ21}
& \int_{\Gamma_j}\int_{B_j}\tilde w_j(x,t)\psi_j^2(x)\eta_j^2(t)dxdt \le \frac{2^{j+2}}{\tilde k}\int_{\Gamma_j}\int_{B_j} w_j^2(x,t)dxdt.  
\end{align}
Without loss of generality, it may be assumed that $x_0=0$. 
Recalling \eqref{subdeltaspace}, we have, for $x\in\supp\psi_j$ and $y\in\R^n\setminus B_j$, 
\[
\frac{1}{|x-y|} = \frac{1}{|y|}\frac{|y|}{|x-y|}\le \frac{1}{|y|}\frac{|x|+|x-y|}{|x-y|} \le \frac{1}{|y|}\left(1 + \frac{r}{\delta_j}\right)\le 2^{j+4}\frac{1}{|y|}. 
\]
Thus 
\begin{align}\label{subJ22}
 &\sup_{\substack{t\in \Gamma_j\\x\in\supp\psi_j}}\int_{\R^n\setminus B_j}\frac{(\tilde w_j)_+(y,t)dy}{|x-y|^{n+2s}} \\
 &\qquad\le 2^{(j+4)(n+2s)}\sup_{t_0-r^{2s}<t<t_0}\int_{\R^n\setminus B_{r/2}}\frac{(w_0)_+(y,t)dy}{|y|^{n+2s}}\notag\\
 & \qquad\le \frac{2^{(j+4)(n+2s)}}{r^{2s}}\tail_\infty(w_0;x_0,r/2,t_0-r^{2s},t_0). \notag
\end{align}
From \eqref{subJ21} and \eqref{subJ22} we conclude 
\begin{align}\label{subJ2}
&J_2 \le \frac{2^{(j+4)(n+2s)}}{\delta r_j^{2s}}\int_{\Gamma_j}\int_{B_j} w_j^2(x,t)dxdt.
\end{align}
From \eqref{subeqphi} we get 
\begin{equation}\label{subJ3}
J_3 \le \frac{C2^{2sj}}{r_j^{2s}}\int_{\Gamma_j}\int_{B_j}w_j^2dxdt. 
\end{equation}
Using the estimates \eqref{subJ1}, \eqref{subJ2} and \eqref{subJ3} for $J_1$, $J_2$ and $J_3$ in \eqref{subI1I2}, we find,  
\begin{align}\label{subI1I2final}
I_1 + I_2& \le \frac{C2^{(n+2s)j}}{\delta}r_j^{-2s}\int_{\Gamma_{j}}\int_{B_j} w_j^{2}(x,t)dxdt.
\end{align}
Similarly to the inequality \eqref{tilde}, we have 
\begin{equation}\label{tilde2}
\tilde w_j^{2\kappa}\ge (k_{j+1}-\tilde k_j)^{2(\kappa-1)}w_{j+1}^2 = \left(2^{-j-2}\tilde k\right)^{2(\kappa-1)}w_{j+1}^2. 
\end{equation}
We now estimate the left hand side of \eqref{subsobolevUj+1} with \eqref{tilde2} and its right hand side with \eqref{subI1I2final}. This yields
\begin{align*}
&\left(2^{-j-2}\tilde k\right)^{2(\kappa-1)}\fint_{\Gamma_{j+1}}\fint_{B_{j+1}}|w_{j+1}|^{2}dxdt \\
&\le r_{j}^{-2s}C
\frac{2^{(n+2s)j}}{\delta}\int_{\Gamma_{j}}\fint_{B_j} w_j^{2}dxdt\left(\frac{2^{(n+2s)j}}{\delta}r_j^{-2s}\int_{\Gamma_{j}}\fint_{B_j} w_j^{2}dxdt\right)^{\frac{2s}{n}}\\
&\le C\left(\frac{2^{(n+2s)j}}{\delta}\fint_{\Gamma_{j}}\fint_{B_j} w_j^{2}dxdt\right)^{\frac{n+2s}{n}}.
\end{align*}
Let 
\[
A_j = \left(\fint_{U_j}w_j^2dxdt\right)^{\frac12}. 
\]
Then 
\begin{equation*}
\frac{A_{j+1}}{\tilde k}\le C\frac{\alpha^j}{\delta^\kappa}\left(\frac{A_j}{\tilde k}\right)^\kappa. 
\end{equation*}
Lemma \ref{lemma_iteration}, with $Y_j = A_j/\tilde k$ and $\beta = \kappa-1 = 2s/n$, says that $\lim_jA_j=0$ if 
\begin{equation}\label{A0}
\frac{A_0}{\tilde k} \le \left(\frac{\delta^\kappa}{C}\right)^{\frac{n}{2s}}\alpha^{-\frac{n^2}{2s}}. 
\end{equation}
Whence we see that if $C = C(n,s)$ is large enough, the choice 
\[
\tilde k = C\delta^{-\frac{n+2s}{2s}}\left(\fint_{U^+(r)}u_+^2dxdt\right)^{\frac12} + \delta \tail_\infty(u_+;x_0,r/2,t_0-r^{2s},t_0)
\]
guarantees that both \eqref{kge} and \eqref{A0} hold. 
It follows that $(u-\tilde k)_+ = 0$ in $U(r)$, which proves the proposition.


\end{proof}

\subsection{Proof of Theorem \ref{thm_localbound} and \ref{thm_localboundfinal}}


\begin{proof}
Let $0<\rho<r$ and assume that $r^{2s}<t_0<T$. From Proposition \ref{prop_localbound} and Young's inequality we obtain, for any $\gamma\in[1/2,1)$, 
\begin{align}\label{bdd1}
\sup_{U^-(\gamma \rho)}u &\le  \frac{C}{(1-\gamma)^{\frac{n+2s}{2}}}\delta^{-\frac{n+2s}{2s}}\left(\fint_{U^-(\rho)}u_+^2dxdt\right)^{\frac12}\\
&+ \delta \tail_\infty(u_+;x_0,\rho/2,t_0-\rho^{2s},t_0)\notag \\
&\le \frac12\sup_{U^-(\rho)} + \frac{C}{(1-\gamma)^{m_2}}\delta^{-m_1}\fint_{U^-(\rho)}u_+dxdt\notag\\
& + \delta \tail_\infty(u_+;x_0,\rho/2,t_0-\rho^{2s},t_0),\notag 
\end{align}
where $m_i = m_i(n,s)>0$, $i=1,2$. 
For any $\sigma \in(1/2,1]$, choose $\gamma$ so that $\theta = \sigma\gamma\ge 1/2$. Upon replacing $\rho$ by $\sigma \rho$ in \eqref{bdd1}, we find that 
\begin{align}
\sup_{U^-(\theta \rho)}u &\le \frac12\sup_{U^-(\sigma \rho)}u + \frac{C\sigma^{m_2}}{(\sigma-\theta)^{m_2}}\delta^{-m_1}\fint_{U^-(\sigma \rho)}u_+dxdt\\
&\quad + \delta \tail_\infty(u_+;x_0,\sigma\rho/2,t_0-(\sigma\rho)^{2s},t_0)\notag\\
& \le \frac12\sup_{U^-(\sigma \rho)}u + \frac{C}{(\sigma-\theta)^{m_2}}\delta^{-m_1}\fint_{U^-(\rho)}u_+dxdt\notag\\
&\quad + 2^{2s}\delta \tail_\infty(u_+;x_0,\rho/4,t_0-\rho^{2s},t_0)\notag . 
\end{align}
An application of Lemma \ref{lemma_HL} gives 
\begin{align}\label{sup1}
\sup_{U^-(\theta \rho)}u& \le \frac{C}{(\sigma-\theta)^{m_2}}\delta^{-m_1}\fint_{U^-(\rho)}u_+dxdt\\
& + \delta \tail_\infty(u_+;x_0,\rho/4,t_0-\rho^{2s},t_0).  \notag
\end{align}
By Lemma \ref{tail_inf_le_tail}
\begin{align}\label{tail1}
\tail_\infty(u_+;x_0,\rho/4,t_0-\rho^{2s},t_0) & \le  C\e^{-1}\tail(u_+;x_0,\rho/4,t_0-(1+\e) \rho^{2s},t_0)\\
&+ C\e^{-1}\fint_{t_0-(1+\e)\rho^{2s}}^{t_2}\fint_{B_\rho(x_0)}u_+dxdt. \notag
\end{align}
Let $r = (1+\e)^{\frac{1}{2s}}\rho$ and let $\beta = (1+\e)^{-\frac{1}{2s}}$. Then 
\[
\e^{-1}  = (\beta^{-2s}-1)^{-1} = \beta^{2s}(1-\beta^{2s})^{-1} \le \frac{1}{1-\beta^{2s}}. 
\]
It can be checked using elementary calculus that there exists $m_3(s)>0$ such that $1-\beta^{2s}\ge (1-\beta)^{m_3}$. Thus, from \eqref{sup1}, with $\sigma = 1$ and \eqref{tail1} we get 
\begin{align*}
\sup_{U^-(\theta \rho)}u& \le \frac{C}{(1-\theta)^{m_2}}\delta^{-m_1}\fint_{U^-(\rho)}u_+dxdt\\
& + \delta \frac{C}{(1-\beta)^{m_3}}\tail(u_+;x_0,\rho/4,t_0- r^{2s},t_0)\\
&+ \delta \frac{C}{(1-\beta)^{m_3}}\fint_{t_0-r^{2s}}^{t_2}\fint_{B_\rho(x_0)}u_+dxdt. 
\end{align*}
We now set, for $\lambda\in(1/2,1)$, $\theta = \beta = \sqrt{\lambda}$, so that $\theta \rho = \theta\beta r = \lambda r$. Using that 
\[
\frac{1}{1-\sqrt\lambda} = \frac{1+\sqrt\lambda}{1-\lambda}\le \frac{2}{1-\lambda},  
\]
we get 
\begin{align}\label{kuk}
&\sup_{U^-(\lambda r)}u \le \frac{C}{(1-\lambda)^{m}}\delta^{-m}\fint_{U^-(r)}u_+dxdt\\
& + \frac{C\delta}{(1-\lambda)^{m}}\tail(u_+;x_0,\rho/4,t_0- r^{2s},t_0)\notag\\
& \le \frac{C}{(1-\lambda)^{m}}\delta^{-m}\fint_{U^-(r)}u_+dxdt\notag \\
& + \frac{C\delta}{(1-\lambda)^{m}}\int_{t_0-r^{2s}}^{t_0}\int_{B_r\setminus B_{\rho/4}}\frac{u_+dxdt}{|x-x_0|^{n+2s}}\notag\\
& + \frac{C\delta}{(1-\lambda)^{m}}\tail(u_+;x_0,r,t_0- r^{2s},t_0)\notag\\
&\le \frac{C}{(1-\lambda)^{m}}\delta^{-m}\fint_{U^-(r)}u_+dxdt + 
\frac{C\delta}{(1-\lambda)^{m}}\tail(u_+;x_0,r,t_0- r^{2s},t_0).\notag
\end{align}
where $m = \max_{i=1,2,3}m_i$. We now assume the hypothesis of Theorem \ref{thm_localboundfinal}.  
Letting $\delta = (1-\lambda)^mM^{-1}$ and employing Lemma \ref{lemma_tail+-}, we find 
\begin{align*}
\sup_{U^-(\lambda r)}u& \le \frac{CM^m}{(1-\lambda)^{2m}}\fint_{U^-(r)}u_+dxdt\\
& + \frac{C}{M}\sup_{U^-(r)}u_+ + \frac{C}{M}\tail(u_-;x_0,r,t_0- r^{2s},t_0). 
\end{align*}
Setting $M = C\eta^{-1}$ for $\eta\in(0,1)$ and using once more Lemma \ref{lemma_HL}, as well as the assumption on the positivity of $u$, we arrive at 
\begin{align*}
\sup_{U^-(\lambda r)}u& \le \frac{C\eta^{-m}}{(1-\lambda)^{2m}}\fint_{U^-(r)}u_+dxdt\\
& + \eta\left(\frac{r}{R}\right)^{2s}\tail(u_-;x_0,R,t_0- r^{2s},t_0). 
\end{align*}
This proves Theorem \ref{thm_localboundfinal}. In the case of Theorem \ref{thm_localbound}, we proceed in the same way from inequality  \eqref{kuk}, but without using Lemma \ref{lemma_tail+-}.  
\end{proof}

\section{The Harnack Inequality}\label{sec_Harnack}

\subsection*{Proof of Theorem \ref{Harnack}}
\begin{proof}
Let $r/2\le r_2<r$. From Theorem \ref{thm_localboundfinal}, with $\delta=1$,  we get,   
\begin{align}\label{suprhalf}
&\sup_{U^-(t_0,r/2)}u  \le C\fint_{U^-(t_0,r_2)}udxdt + \left(\frac{r}{R}\right)^{2s}
\tail\left(u_-;x_0,R,t_0-r_2^{2s},t_0\right),  
\end{align}
where $C = C((r_2-r/2)/r)$. 
Let $r/2\le r_1<r$, let $T_0 = t_0+r^{2s}-r_2^{2s}$ and suppose that 
\[
u\ge 0\text{ in }B_R\times(T_0-r^{2s},T_0+r^{2s}).   
\]
From \eqref{whr1r2} in the proof of the weak Harnack inequality, we have 
\begin{align}\label{r1r2}
&\fint_{B_{r_2}\times(T_0-r^{2s},T_0-r^{2s}+r_2^{2s})}udxdt
\le C\inf_{B_{r_1}\times(T_0+r^{2s}-r_1^{2s},T_0+r^{2s})}u \\
&\qquad + C\left(\frac{r}{R}\right)^{2s}\tail_\infty(u_-;x_0,R,T_0-r^{2s},T_0+r^{2s}),\notag
\end{align}
where $$C = C\left(\frac{r-r_1}{r},\frac{r-r_2}{r}\right).$$ 
Set $r_1 = r/2$ and choose $r_2$ so that $r_2^{2s}=\alpha(r/2)^{2s}$ for some $1<\alpha<2^{2s}$.  
Then \eqref{r1r2} reads 
\begin{align}\label{r1r2final}
&\fint_{U^-(t_0,r_2)}udxdt
\le C\inf_{B_{r/2}\times(t_0+2r^{2s}-(1+\alpha)(r/2)^{2s},t_0+2r^{2s} -\alpha(r/2)^{2s})}u \\
&\qquad + C\left(\frac{r}{R}\right)^{2s}\tail_\infty(u_-;x_0,R,t_0-\alpha\left(r/2\right)^{2s},t_0+2r^{2s} - \alpha\left(r/2\right)^{2s}), \notag
\end{align}
where $C = C(\alpha)$. 
Let 
\[
t_1 = t_0+2r^{2s}-\alpha(r/2)^{2s}. 
\]
By Corollary \ref{tail_inf_le_tail_minus} and the fact that $u\ge 0$ in 
$B_R\times(t_0-r^{2s},t_1)$, 
\[
\tail_\infty\left(u_-;x_0,R,t_0-\alpha(r/2)^{2s},t_1\right)\le C(\alpha)\tail\left(u_-;x_0,R,t_0-r^{2s},t_1\right). 
\]
Hence we obtain from \eqref{r1r2final} and \eqref{suprhalf}, 
\begin{align*}
\sup_{U^-(t_0,r/2)} \le C\inf_{U^-(t_1,r/2)}u + C\left(\frac{r}{R}\right)^{2s}\tail(u_-;x_0,R,t_0-r^{2s},t_1), 
\end{align*}
with $C = C(\alpha)$. This completes the proof. 
\end{proof}

\bibliography{parabolicHarnack}{}
\bibliographystyle{plain}

\end{document}